\documentclass[reqno,10.5pt]{amsart}
\usepackage{amsaddr}
\usepackage[numbers,sort]{natbib}
\usepackage[right=1in]{geometry}

\usepackage{lineno}
\usepackage{graphicx}
\usepackage{amsmath,amsthm,amssymb}
\usepackage{xargs}
\usepackage{xifthen}
\usepackage{url}
\usepackage{subfigure}
\usepackage{mathrsfs}
\usepackage{bbold}
\usepackage{enumitem}
\setdescription{leftmargin=\parindent,labelindent=\parindent,font=\itshape}
\usepackage[T1]{fontenc}
\usepackage{calligra}

\usepackage[usenames,dvipsnames]{color}

\DeclareMathAlphabet{\mathpzc}{OT1}{pzc}{m}{it}

\theoremstyle{plain}
\newtheorem{thm}{Theorem}
\newtheorem{prop}{Proposition}

\newtheorem{lem}{Lemma}
\newtheorem{obs}{Observation}

\theoremstyle{remark}
\newtheorem{rem}{Remark}
\newtheorem{examp}{Example}

\theoremstyle{definition}

\newtheorem{defn}{Definition}

\newtheoremstyle{claimstyle}{1.5pt}{}{\normalfont}{}{\itshape}{. }{ }{\thmname{\qquad #1}\thmnumber{ #2}}
  
\theoremstyle{claimstyle}
\newtheorem{claim}{Claim}

\DeclareMathOperator{\co}{conv}
\DeclareMathOperator{\signature}{sign}
\DeclareMathOperator{\previous}{prev}
\DeclareMathOperator{\nextid}{next}

\newcommand{\prev}[1]{\previous(#1)}
\newcommand{\next}[1]{\nextid(#1)}

\newcommand{\conv}[1][]{%
\ifthenelse{\isempty{#1}}%
{\co}
{\co #1}
}

\renewcommand{\r}{r}
\newcommand{\bigO}{\mathcal{O}}

\newcommand{\bin}{0\backslash 1}

\newcommand{\ints}{\mathbb{Z}}
\newcommand{\kpset}{K}
\newcommand{\rhs}{b}
\newcommand{\wt}{a}
\newcommand{\ub}{u}
\newcommand{\base}{\alpha}
\newcommand{\Opt}[1][c]{\Omega(#1)}
\newcommand{\x}{\hat{x}}
\renewcommand{\ell}{\mathscr{L}}
\renewcommand{\L}{\Opt}

\renewcommand{\preceq}{\preccurlyeq}
\renewcommand{\succeq}{\succcurlyeq}

\newcommand{\kp}{\ensuremath{\kpset}}
\newcommand{\kpone}{\ensuremath{\kpset^{1}}} 
\newcommandx{\expkp}[2][1=\base,2=\rhs,usedefault]{\ensuremath{\kpset_{#1}(#2)}}
\newcommand{\kpgen}{\kp^{\mathrm{gen}}}

\newcommand{\cokp}{\conv{\kp}}
\newcommandx{\cokpone}[2][1=n,2=\rhs,usedefault]{\conv\kpone[#1][#2]}	

\newcommand{\bool}[1][]{%
\ifthenelse{\isempty{#1}}%
{\{0,1\}}
{\{0,1\}^{#1}}
}

\newcommand{\zerovec}{\ensuremath{\mathbf{0}}}

\newcommand{\onevec}[1][]{%
\ifthenelse{\isempty{#1}}%
{\mathbf{e}}
{\mathbf{e}_{#1}}
}

\newcommandx{\cube}[2][1=\base,2=\rhs,usedefault]{\mathcal{H}(#2;#1)}
\newcommandx{\rect}[2][1=n,2=\ub,usedefault]{\mathcal{H}_{#1}(#2)}

\newcommandx{\sign}[3][1=\base,2=\rhs,3=,usedefault]{%
\ifthenelse{\isempty{#3}}
{\signature(#2;#1)}
{\signature_{#3}(#2;#1)}
}

\newcommand{\myqed}{\ensuremath{\diamond}}
\renewcommand{\i}{i^{*}}

\newcommand{\blue}[1]{#1}
\newcommand{\green}[1]{#1}

\begin{document}

\title{Convex hulls of superincreasing knapsacks and lexicographic orderings}
\author{Akshay Gupte}
\address{Department of Mathematical Sciences, Clemson University} 
\email{agupte@clemson.edu}

\begin{abstract}
We consider bounded integer knapsacks where the weights and variable upper bounds together form a superincreasing sequence. The elements of this superincreasing knapsack are exactly those vectors that are lexicographically smaller than the greedy solution to optimizing over this knapsack. We describe the convex hull of this $n$-dimensional set with $\bigO(n)$ facets. We also establish a distributive property by proving that the convex hull of $\le$- and $\ge$-type superincreasing knapsacks can be obtained by intersecting the convex hulls of $\le$- and $\ge$-sets taken individually. Our proofs generalize existing results for the $\bin$ case.
\end{abstract}

\newcommand{\sep}{, }

\maketitle


\section{Introduction}\label{sec:intro}
Given positive integers $n,\rhs$, and $(\wt_{i},\ub_{i})$ for all $i \in N :=\{1,\ldots,n\}$, we consider a bounded integer knapsack defined as $\kp := \{x \in \ints^{n}_{+} \colon \wt^{\top} x \,\le\, \rhs, \, 0 \le x_{i} \le \ub_{i} \ i=1,\dots,n \}$. 
Without loss of generality (w.o.l.o.g.) we assume that $\wt_{i}\ub_{i} \le \rhs \ \forall i\in N$ and that $\wt^{\top}\ub > \rhs$ to ensure a nontrivial set. When all the upper bounds are equal to one, we have the $\bin$ knapsack $\kpone$. The convex hull of $\kp$, denoted by $\cokp$, is referred to as the knapsack polytope. 


The study of the knapsack polytope has received considerable attention in literature and in general, there may exist exponentially many facet-defining inequalities. There exist special classes of $\kp$ for which a complete description of $\cokp$ is known. For the $\bin$ knapsack, these results include the minimal covers of \citet{wolsey1975faces} assuming certain matroidal properties for $\wt^{\top}x \le \rhs$, $(1,k)$-configurations of \citet{padberg19801}, weight-reduction principle of \citet{weismantel19970} when $\wt_{i}\in\{1, \lfloor \rhs/3 \rfloor + 1, \lfloor \rhs/3 \rfloor + 2, \dots, \lfloor \rhs/2 \rfloor \} \ \forall i$ or $\wt_{i}\in\{1, \lfloor \rhs/2 \rfloor + 1, \lfloor \rhs/2 \rfloor + 2, \dots, \rhs \}\ \forall i$, and \citet{weismantel1996hilbert} when $\wt_{i} \in \{\acute{\wt},\breve{\wt}\} \ \forall i$ and for two distinct positive integers $\acute{\wt}$ and $\breve{\wt}$. There also exist complete descriptions of $\cokp$ for upper bounds not equal to 1. For a \emph{divisible} knapsack, i.e. when $\wt_{i-1}\, |\, \wt_{i}$ for all $i\ge 2$, three results are known: (i) \citet{marcotte1985cutting} when $\kp = \{x \in \ints^{n}_{+} \colon \wt^{\top} x \le \rhs \}$, (ii) \citet{pochet1995integer} when $\kp = \{x \in \ints^{n}_{+} \colon \wt^{\top} x \ge \rhs \}$, and (iii) \citet{pochet1998sequential} when $\kp = \{x \in \ints^{n}_{+} \colon \wt^{\top} x \le \rhs, \zerovec \le x \le \ub \}$. Recently, \citet{Cacchiani201374} described the convex hulls of $\le$- and $\ge$-type knapsacks with a generalized upper bound constraint $\sum_{i\in N}x_{i}\le 2$. The polytopes in \citep{pochet1995integer,pochet1998sequential} involve an exponential number of valid inequalities, whereas the polytopes in \citep{Cacchiani201374,marcotte1985cutting} have $\bigO(n)$ facets. 



In this paper, we are interested in the convex hull of a special class of $\kp$ characterized as follows.




\begin{defn}[Superincreasing knapsack]\label{def:superinc}
The set $\kp$ is said to be a \emph{superincreasing knapsack} if $\{(\wt_{i},\ub_{i})\}_{i\in N}$ forms a weakly superincreasing sequence of tuples, i.e. $\sum_{k=1}^{i}\wt_{k}\ub_{k}  \le  \wt_{i+1} \ \forall i \ge 1$.
\end{defn}

$\bin$ superincreasing knapsacks have been historically used in cryptographic systems \citep{merkle1978hiding,odlyzko1990rise}; their structure and linear time complexity were  investigated \blue{in \citep{shamir1984polynomial}}. \citet{atkinson1990sums} obtained a wider class of $\bin$ knapsacks with linear time complexity. \green{From the viewpoint of a polyhedral study}, the knapsack polytope for superincreasing $\kpone$ was first studied by \citet{laurent1992characterization}, whose result is paraphrased below. 

\begin{thm}[\citep{laurent1992characterization}]\label{thm:laurent}
For any positive integer $\rhs$, the $\bin$ knapsack polytope $\cokpone$ is completely described by its minimal cover inequalities if and only if $\{(\wt_{i},1)\}_{i\in N}$ forms a weakly superincreasing sequence. Furthermore, all the $\bigO(n)$ minimal covers can be explicitly enumerated.
\end{thm}

\newcommand{\explen}[1][]{\lfloor \log_{\base}{\ub_{#1}}\rfloor}
We extend the sufficiency condition of Theorem~\ref{thm:laurent} to the case when the variable upper bounds in $\kp$ are not necessarily equal to one. We explicitly describe the convex hull  with $\bigO(n)$ nontrivial facets. As is to be expected, the proposed inequalities reduce to minimal covers of $\kpone$ when $u$ is a vector of ones. The convex hull proof for $\kpone$ is simpler since the coefficient matrix for system of minimal covers is an interval matrix and hence totally unimodular.
Since \citet{marcotte1985cutting} describes the convex hull of divisible $\kp$ using $\bigO(n)$ facets, it is obvious that the superincreasing property is not a necessary condition for $\cokp$ to have $\bigO(n)$ facets. 
Besides generalizing the result of \citeauthor{laurent1992characterization}, another motivation for studying superincreasing knapsacks is that such sets appear after reformulating the integer variables in a mixed integer program; see \citet{siampaper}. The most common example of such reformulations is the set 
\begin{equation}\label{expkp}
\expkp := \left\{\zeta\colon  \sum_{t} \base^{t-1}\zeta_{t} \le \ub,\; \zeta_{t}\in\{0,1,\ldots,\base-1\} \ \: \forall t=1,\ldots,\explen+1\right\}
\end{equation}
obtained after $\base$-nary expansion of a integer variable: $x = \sum_{t=1}^{\explen+1} \base^{t-1}\zeta_{t}$. Convex hull of $\expkp[2]$ was independently studied by \citep{gillmann2006revlex,siampaper}. A complete knowledge of the superincreasing knapsack polytope will provide a family of valid inequalities to the mixed integer program. \citeauthor{siampaper} demonstrated the practical usefulness of facets to binary expansion knapsacks as cutting planes in a branch-and-cut algorithm for solving mixed integer bilinear programs. 

\begin{rem}
The extended formulation of $\kp$, obtained after adding new variables $\zeta_{it}\in \{0,1\} \: \forall i,t$ and basis expansion of each $x_{i}$ as $x_{i} = \sum_{t=1}^{\explen[i]+1} \base^{t-1}\zeta_{it}$ for some $\base\in\ints_{++}$, does not obey the superincreasing property. Hence we cannot obtain $\cokp$ simply as a projection of the extended formulation.  
\end{rem}

Note that there is no inclusive relationship between superincreasing and divisible knapsacks. However, certain types of knapsacks, such as $\expkp$, may be both divisible and superincreasing. For divisible superincreasing knapsacks, our result provides a explicit linear size minimal description as compared to the implicit exponential size description in \citep{pochet1998sequential}.

Throughout this paper, we assume that $\kp$ is superincreasing. We begin by analyzing the greedy solution of $\kp$ in Section \ref{sec:complex} and use it to provide a useful geometric interpretation to our assumption of superincreasing tuples $\{(\wt_{i},\ub_{i}) \}_{i\in N}$. Section \ref{sec:conv} derives a set of facet-defining inequalities, referred to as \emph{packing inequalities}, to $\cokp$ and our first main result in Theorem~\ref{thm:conv} proves that these inequalities describe $\cokp$. In Section \ref{sec:twoside}, we prove that the convex hull of intersection of two superincreasing knapsacks is given by the facets of the individual knapsack polytopes. This second main result in Theorem~\ref{thm:twokp2} is indeed interesting since the convex hull operator does not distribute in general and further implies that the convex hull of a family of $m$ intersecting superincreasing knapsacks is described by $\bigO(n)$ linear inequalities. For general $\bin$ knapsacks, new valid inequalities were derived in \citep{martin1998intersection,louveaux2008polyhedral} for intersection of two $\le$-type knapsacks and in \citep{louveaux2008polyhedral,fernandez1994partial} for one $\le$- and one $\ge$-type knapsack.

We adopt the following notation. $\conv{\mathcal{X}}$ is the convex hull of a set $\mathcal{X}$. $\ints_{+} (\blue{\ints_{++}})$ is the set of nonnegative (\blue{positive}) integers. $\onevec$ is a vector of ones, $\onevec[i]$ is the $i^{th}$ unit vector and $\zerovec$ is a vector of zeros. $\rect := \{x \in \ints_{+}^{n}\colon x_{i}\le \ub_{i}\ \forall i\}$ is a discrete hyper-rectangle. For $\mathpzc{l} > \mathpzc{s}$, we denote $\sum_{\mathpzc{l}}^{\mathpzc{s}}(\cdot) = 0$ and $\prod_{\mathpzc{l}}^{\mathpzc{s}}(\cdot) = 1$. The positive part of $\xi\in\Re$ is denoted by $[\xi]^{+} := \max\{0,\xi\}$. 

\newcommandx{\B}[2][1=\wt,2=\ub,usedefault]{\mathcal{B}(#1,#2)}
\renewcommand{\b}{\mathpzc{\rhs}}
\newcommand{\I}{I}
\newcommand{\Ij}[1][j]{\I_{#1}}
\newcommand{\Ijm}[1][j]{\I_{#1}^{-}}
\newcommand{\val}[1][n]{f^{*}_{#1}(c)}
\newcommandx{\rhsmax}[3][1=\wt,2=\ub,3=\rhs,usedefault]{\mathpzc{g}(#1,#2,#3)} 

\section{Structure of $\kp$}\label{sec:complex}
\green{This section discusses structural properties of a superincreasing knapsack - first we present a geometric interpretation to the algebraic requirements of Definition~\ref{def:superinc}, then we characterize maximal packings of $\kp$ and finally we state a dynamic program to optimize over $\kp$.} \blue{The proposed results, especially the maximal packing and dynamic program, are known in literature for the $\bin$ case, see for example \citet{shamir1984polynomial}, which has important applications in cryptographic systems. Our contribution is to extend these results to the general integer case and establish a foundation for our main theorems in Sections~\ref{sec:conv} and \ref{sec:twoside}.}

The notion of \emph{lexicographic ordering} will be useful for the rest of the paper. For any two vectors $v^{1}$ and $v^{2}$, the vector $v^{1}$ is lexicographically smaller than $v^{2}$, denoted as $v^{1}\preceq v^{2}$, if either $v^{1}=v^{2}$ or the first (in reverse order) nonzero element $i$ of $v^{1}-v^{2}$ is such that $v^{1}_{i} < v^{2}_{i}$. In the latter case, we denote $v^{1}\prec v^{2}$. Since $\preceq$ is a total order, for any distinct $v^{1}$ and $v^{2}$, either $v^{1}\prec v^{2}$ or $v^{2}\prec v^{1}$ (equivalently $v^{1}\succ v^{2}$). 

\green{The \emph{greedy solution} of an arbitrary (not necessarily superincreasing) knapsack $\kpgen :=\{x\in\rect[][\tilde{\ub}]\colon \tilde{\wt}^{\top}x \le \tilde{\rhs} \}$ is given by}
\begin{equation}\label{greedysol}
\tilde{\theta}_{i} := \min\left\{\tilde{\ub}_{i}, \left\lfloor \frac{\tilde{\rhs} - \sum_{k=i+1}^{n}\tilde{\wt}_{k}\tilde{\theta}_{k}}{\tilde{\wt}_{i}}\right\rfloor\right\} \qquad \forall i=n,\ldots,1. 
\end{equation}
\citet{magazine1975greedy} referred to this solution in the case of trivial upper bounds, i.e., when $\tilde{\ub}_{i} = \lfloor \tilde{\rhs}/\tilde{\wt}_{i} \rfloor \ \forall i$, and studied conditions under which it is the optimal solution for maximizing over $\kpgen$. By construction, we have $\tilde{\theta}\in\kpgen$. \blue{In fact, $\tilde{\theta}$ is lexicographically the largest vector in $\kpgen$.
\begin{lem}\label{lem:greedy}
$\kpgen \subseteq \{x\in\rect[][\tilde{\ub}]\colon x \preceq\tilde{\theta}\}$.
\end{lem}
\begin{proof}
Suppose there exists some $x\in\kpgen$ with $x\succ \tilde{\theta}$. Let $\i := \max\{i \in N\colon x_{i}\neq \tilde{\theta}_{i}\}$. Since $x\succ \tilde{\theta}$ by assumption and $x,\tilde{\theta}\in\ints^{n}$, we have $x_{\i} \ge \tilde{\theta}_{\i}+1$. Now $x_{\i}\le\tilde{\ub}_{\i}$ and equation \eqref{greedysol} imply that $\tilde{\theta}_{\i} = \lfloor \frac{\tilde{\rhs} - \sum_{i>\i}\tilde{\wt}_{i}\tilde{\theta}_{i}}{\tilde{\wt}_{\i}} \rfloor$. Hence $x_{\i} > (\tilde{\rhs} - \sum_{i>\i}\tilde{\wt}_{i}\tilde{\theta}_{i})/\tilde{\wt}_{\i}$. Then $\tilde{\wt}^{\top}x \ge \sum_{i>\i}\tilde{\wt}_{i}x_{i} + \tilde{\wt}_{\i}x_{\i} > \sum_{i>\i}\tilde{\wt}_{i}\tilde{\theta}_{i} + \tilde{\rhs} - \sum_{i>\i}\tilde{\wt}_{i}\tilde{\theta}_{i}= \tilde{\rhs}$, a contradiction to $x\in\kpgen$.
\end{proof}
}

Henceforth, we denote $\theta$ to be the greedy solution of a superincreasing knapsack $\kp$. \green{Proposition~\ref{prop:lexord2} states that for superincreasing knapsacks, the inclusion in Lemma~\ref{lem:greedy} becomes an equality. The proof of this depends on the following observation about points in $\kp$.}

{
\renewcommand{\gamma}{y}
\renewcommand{\theta}{x}
\begin{lem}\label{lem:lexord}
Let $x \in \kp$ and $y\in\rect$ be such that $y \preceq x$. Then $\wt^{\top}y \le \wt^{\top}x$ and hence $y \in \kp$. \green{If $y\prec x$ and $\i := \max\{i \in N\colon y_{i}\neq x_{i}\}$, we have}
\begin{enumerate}
\green{
\item $\wt^{\top}y = \wt^{\top}x$ if and only if $\wt_{\i}=\sum_{i<\i}\wt_{i}\ub_{i}$, $\gamma_{\i}=\theta_{\i}-1$, and $\gamma_{i}=\ub_{i}, \theta_{i} = 0 \ \forall i < \i$.
\item $(\ub_{1},\ldots,\ub_{\i-1},y_{\i},x_{\i+1},\ldots,x_{n}) \in \kp$.}
\end{enumerate}
\end{lem}
\begin{proof}
Suppose that $y\neq x$. Since $y\preceq x$ and $y\in\ints^{n}_{+}$, we have $y_{\i}\le x_{\i}-1$. Then, $\wt^{\top}(\gamma - \theta) = \sum_{i<\i}\wt_{i}(\gamma_{i}-\theta_{i}) +\wt_{\i}(\gamma_{\i}-\theta_{\i}) \le \sum_{i< \i}\wt_{i}\ub_{i}\: - \: \wt_{\i} \le 0$, giving us $\wt^{\top}\gamma \le \rhs$ and $\gamma\in\kp$. This also leads to the conditions for $\wt^{\top}(\gamma - \theta) = 0$. Finally, $(\ub_{1},\ldots,\ub_{\i-1},y_{\i},x_{\i+1},\ldots,x_{n})\prec x$ implies the second claim.
\end{proof}
}


\blue{Lemmas~\ref{lem:greedy} and \ref{lem:lexord} and the fact that $\theta\in\kp$ gives us}
\begin{prop}\label{prop:lexord2}
$\kp = \{x\in\rect\colon x \preceq\theta \}$.
\end{prop}
Thus, a superincreasing knapsack is exactly the set of integer points within $\rect$ that are lexicographically smaller than the greedy solution. This implies that two distinct superincreasing sequences, $\{\wt,\ub\}$ and $\{w,\ub\}$, represent the same knapsack if and only if the corresponding greedy solutions are equal. 

We now describe the connection between the greedy solution and the notion of \emph{maximal packing}. Let $\rhsmax[\tilde{\wt}][\tilde{\ub}][\tilde{\rhs}] := \max\{\tilde{\wt}^{\top}x \colon \tilde{\wt}^{\top}x \le \tilde{\rhs}, x\in \rect[][\tilde{\ub}] \}$ denote the maximum attainable capacity of $\kpgen$. Clearly, $\kpgen = \{x\in\rect[][\tilde{\ub}]\colon \tilde{\wt}^{\top}x \le \rhsmax[\tilde{\wt}][\tilde{\ub}][\tilde{\rhs}] \}$. A maximal packing of $\kpgen$ is a vector $x \in\kpgen$ such that $\tilde{\wt}^{\top}x = \rhsmax[\tilde{\wt}][\tilde{\ub}][\tilde{\rhs}]$. Maximal packing may not be unique and in general, computing it reduces to solving the NP-hard subset-sum problem, although it can be computed in linear time for divisible knapsacks \citep{alfonsin1998variations}. \blue{For superincreasing knapsacks, Proposition~\ref{prop:lexord2} and Lemma~\ref{lem:lexord} imply that $\theta$ is a maximal packing of $\kp$.} The conditions for $\theta$ being the only maximal packing are characterized next.

\begin{prop}\label{prop:packing}
$\theta$ is a unique maximal packing of $\kp$ \green{if and only if for every $j\in N$ with $\wt_{j} = \sum_{i=1}^{j-1}\wt_{i}\ub_{i}$ and $\theta_{j} > 0$, there exists $i < j$ such that $\theta_{i} > 0$.}
\end{prop}
\begin{proof}
Assume that for every $j\in N$ with $\wt_{j} = \sum_{i=1}^{j-1}\wt_{i}\ub_{i}$, there exists $i < j$ such that $\theta_{i} > 0$. Suppose $\theta$ is not the unique maximal packing and there exists some $\gamma\in\kp\setminus\{\theta\}$ such that $\wt^{\top}\gamma = \wt^{\top}\theta$. \blue{Proposition~\ref{prop:lexord2} gives us $\gamma\prec\theta$. Then Lemma~\ref{lem:lexord} implies that $\wt_{\i}=\sum_{i<\i}\wt_{i}\ub_{i}$ and $\theta_{i}=0 \ \forall i<\i$, a contradiction to our assumption.} Now suppose that $\theta$ is unique and let there exist some $j\in N$ with $\wt_{j} = \sum_{i=1}^{j-1}\wt_{i}\ub_{i}$ and \blue{$\theta_{j} > 0$} but \blue{$\theta_{i}=0$} for all $i < j$. Set \blue{$\gamma \prec \theta$ as follows: $\gamma_{i} = \ub_{i} \ \forall i < j$, $\gamma_{j} = \theta_{j}-1$, and $\gamma_{i} = \theta_{i} \ \forall i > j$.} Then $\gamma\in\rect$ and $\wt^{\top}\gamma = \wt^{\top}\theta = \rhsmax$, contradicting the uniqueness of $\theta$.
\end{proof}


\begin{rem}
When $\wt_{i}=\base^{i-1}$ and $\ub_{i}=\base-1$ for all $i\in N$ and some $\base\in\ints_{++}$, we have the set $\expkp$ from \eqref{expkp} and it is straightforward to verify in this case that $\theta$ is the unique representation of $\rhs$ in base $\base$. 
\end{rem}


\blue{\begin{examp}
Let $\wt = (2,8,46,150,310), \ub = (3,5,2,1,2)$ and $\kp = \{x\in\rect[5]\colon \sum_{i=1}^{5}\wt_{i}x_{i} \le 841\}$. It can be verified by enumerating the points in $\kp$ that $\rhsmax[][][841] = 840$. The greedy solution is $\theta=(0,3,1,1,2)$ and observe that $\wt^{\top}\theta = 8(3)+46+150+310(2) = 840$. Also, $\theta$ is the only point in $\kp$ that yields $\rhsmax[][][841]$ and it satisfies the sufficient condition for uniqueness: $\wt_{3} = \wt_{1}\ub_{1}+\wt_{2}\ub_{2}$ with $\theta_{3},\theta_{2} > 0$. An alternate superincreasing knapsack representation is $\kp=\{x\in\rect[5]\colon x_{1} + 3x_{2} + 18x_{3} + 95x_{4}+ 189x_{5}\le 500\}$.

Now let $\kp^{\prime}=\{x\in\rect[5]\colon \sum_{i=1}^{5}\wt_{i}x_{i} \le 863\}$. We have $\rhsmax[][][863] = 862, \theta^{\prime}=(0,0,2,1,2),\wt^{\top}\theta^{\prime} = 862$ and $\theta^{\prime}$ does not satisfy the necessary condition for uniqueness: $\wt_{3} = \wt_{1}\ub_{1}+\wt_{2}\ub_{2}$ with $\theta_{3}^{\prime}>0,\theta_{1}^{\prime}=\theta_{2}^{\prime} = 0$. Another maximal packing is $\gamma^{\prime}=(3,5,1,1,2)$, which is equal to $\theta^{\prime}+(\ub_{1},\ub_{2},-1,0,0)$. 

Finally, let $\tilde{\wt}=(2,8,40,150,310),\tilde{\ub}=(1,5,4,1,2)$ and note that $\tilde{\wt}_{3} < \tilde{\wt}_{1}\tilde{\ub}_{1} + \tilde{\wt}_{2}\tilde{\ub}_{2}$. For $\kpgen = \{x\in\rect[5][\tilde{\ub}]\colon \sum_{i=1}^{5}\tilde{\wt}_{i}x_{i} \le 825 \}$ the greedy solution is $\tilde{\theta} = (1,1,1,1,2)$ with $\tilde{\wt}^{\top}\tilde{\theta}=820 < 822 =  \rhsmax[\tilde{\wt}][\tilde{\ub}][825] = \tilde{\wt}^{\top}\tilde{\gamma}$, where $\tilde{\gamma}=(1,5,4,0,2)$. Hence the greedy solution does not give a maximal packing.
\hfill\myqed
\end{examp}}

Since $\theta$ is a maximal packing of $\kp$ and is computable in $\bigO(n)$ time, we assume w.o.l.o.g. that $\rhs=\rhsmax=\wt^{\top}\theta$. 

\green{The equivalence to lexicographic ordering in Proposition~\ref{prop:lexord2} lends intuition to the points contained in $\kp$ and also enables us to prove our results. An immediate consequence is a linear time algorithm for optimization over $\kp$. To state this result, we first define the support of $\theta$}.
\begin{defn}\label{defn:I}
Let $\I := \{i \in N\colon \theta_{i}\ge 1\}$ be the support of $\theta$ and denote\footnote{$n\in\I$ since $\wt_{n}\ub_{n}\le \rhs$ implies $\theta_{n} = \ub_{n}$ and hence $n$ is the largest index in $\I$.} $\I = \{i_{1},\dots,i_{\r},i_{\r +1}:=n\}$ for some integer $\r \ge 0$, where we assume $i_{1} < i_{2} < \cdots < i_{\r} < n$. For every $j\in N$, let $\Ij := \{i\in \I \colon i > j\}$, $\Ijm := \{i\in \I \colon i < j\}$, $\prev{j} := \max\{i\colon i\in\Ijm \}$, and $\next{j} := \min\{i \colon i\in\Ij \}$. If $j < i_{1}$ (resp. $j=n$), then $\prev{j} = 0$ (resp. $\next{j}=0$).
\end{defn}

\renewcommand{\S}[1]{\mathcal{S}_{#1}}
\renewcommand{\P}[1]{\conv\S{#1}}
\newcommand{\ite}[1][t]{j}

Clearly $\Ij[n] = \Ijm[i_{1}] = \emptyset$ and $\Ij[j+1]=\Ij\setminus \{j+1\}$ for all $j\in N$. 


\begin{prop}\label{prop:dp}
There exists a $\bigO(n)$ time algorithm to optimize over $\kp$. Given any $c\in\Re^{n}$, for every $\ite\in\I$, the optimal value $\val[\ite] := \max \{ \sum_{i=1}^{\ite}c_{i}x_{i} \colon x \in \kp, x_{k}=\theta_{k} \ \forall k\in\Ij[\ite]\}$ is equal to
\begin{equation*}
\val[\ite] \;=\; \max\left\{[c_{\ite}]^{+}(\theta_{\ite}-1) \:+\: \sum_{i=1}^{\ite-1}[c_{i}]^{+}\ub_{i} \;, \; c_{\ite}\theta_{\ite} + \val[{\prev{j}}] \right\}.\qed
\end{equation*}
\end{prop}
\blue{The correctness of this recursion follows from Lemma~\ref{lem:lexord}, Proposition~\ref{prop:lexord2} and the observation that}
\begin{equation}\label{eq:restrict}
\{x\in\kp\colon x_{k}=\theta_{k} \ \forall k\in\Ij[\ite]\} \,\subseteq\, \big\{x\colon x_{j} \in \{0,1,\dots,\theta_{j}\},\, x_{i} = 0 \ \forall i \in \{j+1,\ldots,n\}\setminus \Ij \big\}, \quad \forall j\in N.
\end{equation}
Define $\Opt := \arg\max\{ c^{\top}x \colon x \in \kp\}$ as the set of optimal solutions to the maximization over $\kp$. The dynamic program of Proposition~\ref{prop:dp} can be represented as a binary tree with $|\I|+1$ leaf nodes, as illustrated in Figure~\ref{fig:dp}. The elements of $\Opt$ correspond to some of the leaf nodes whereas the elements of $\I$ are in a bijection to the non-leaf nodes. For a non-leaf $j\in\I$, the set $\{x\in\kp\colon x_{k}=\theta_{k} \ \forall k\in\Ij,x_{j} \le \theta_{j}-1 \}$ contains the leaf descendant denoted by $\ell_{j}$ whereas the set $\{x\in\kp\colon x_{k}=\theta_{k} \ \forall k\in\Ij, x_{j} = \theta_{j} \}$ corresponds to the next non-leaf node $\prev{j}$ if $j\neq i_{1}$. The two leaf descendants of $i_{1}$ are $\ell_{i_{1}}$ and $\ell_{0}$. In particular,
\begin{equation}\label{leafnode}
\ell_{j}=\left\{x \colon x_{i}\in\{0,\ub_{i}\} \ \forall i < j, \, x_{j}\in\{0,\theta_{j}-1 \}, \, x_{i}=\theta_{i} \ \forall i > j \right\} \quad \forall j\in\I, \quad \text{and} \quad \ell_{0} = \{\theta\}.
\end{equation}
This implies that 
\begin{equation}\label{optsol}
\Opt \subseteq \{\theta\} \cup \bigcup_{j\in\I} \ell_{j}.
\end{equation}

\begin{figure}[htbp]
\begin{center}
\includegraphics[scale=0.475]{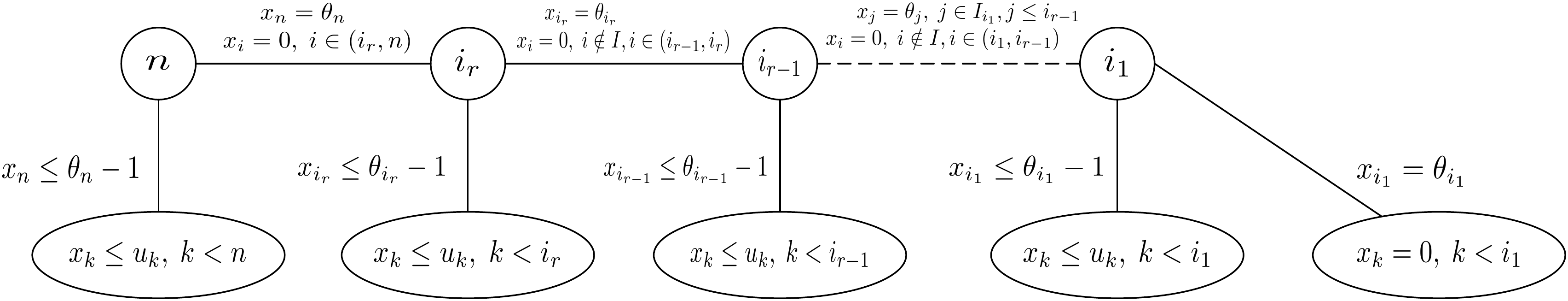}
\caption{Binary tree for the dynamic programming algorithm of Proposition \ref{prop:dp}.}
\label{fig:dp}
\end{center}
\end{figure}

\newcommand{\face}[1][j]{\mathpzc{T}_{#1}}

\section{Facets and convex hull}\label{sec:conv}

\green{We first derive valid inequalities whose coefficients depend on the greedy solution (and maximal packing) $\theta$.} To state the proposed \emph{packing inequalities}, for any $j\in N\setminus n$, define a function $\phi_{j} \colon \Ij \mapsto \ints_{+}$ as
\begin{equation}\label{newphi}
\phi_{j}(i) \::=\: (\ub_{j}-\theta_{j})\,\prod_{\substack{k=\next{j}\colon\\k\in\I}}^{\prev{i}}(\ub_{k}+1-\theta_{k}) \qquad i\in\Ij.
\end{equation}
\green{From notational convention, we have} $\phi_{j}(\next{j}) = (\ub_{j}-\theta_{j})$ and $\phi_{n}(\cdot) = 0$. \green{The recursive definition of $\phi_{j}(\cdot)$ leads to the following identities that will be useful while arguing validity and facet-defining property}. 

\begin{obs}\label{obs:phi}
For any $j\in N\setminus n$ and $i\in\Ij$, we have $\phi_{j}(\next{i}) - \phi_{j}(i) = \phi_{j}(i)(\ub_{i}-\theta_{i})$. Consequently, $\phi_{j}(i) = \ub_{j}-\theta_{j} + \sum_{k\in\Ij\colon k < i}\phi_{j}(k)(\ub_{k}-\theta_{k})$.
\end{obs}
\begin{proof}
\blue{The first statement is obvious from the definition of $\phi_{j}(\cdot)$. The second statement is obtained via a straightforward induction on $i$ and using the first statement.
}
\end{proof}

\begin{prop}[Packing inequalities]\label{prop:valid}
For any $j\in N$, the inequality 
\begin{equation}\label{packineq}
x_{j} \:+\: \sum_{i\in\Ij}\phi_{j}(i)(x_{i}-\theta_{i}) \; \le \; \theta_{j}
\end{equation}
is valid to $\cokp$.
\end{prop}
\begin{proof}
\green{For $j=n$, the inequality is simply $x_{n}\le\ub_{n}$ since $\Ij[n]=\emptyset,\ub_{n}=\theta_{n}$.} We prove validity for $j<n$ by induction on elements of $\Ij$. Denote $\S{j}:=\{x\in\kp\colon x_{k}=\theta_{k}\ \forall k\in\Ij\}$ for every $j\in N$. First we argue that $x_{j}  + (\ub_{j}-\theta_{j})(x_{\next{j}} - \theta_{\next{j}}) \le \theta_{j}$ is valid for $\S{\next{j}}$. For $x\in\S{\next{j}}$, equation~\eqref{eq:restrict} gives us $x_{\next{j}} \le \theta_{\next{j}}$. If $x_{\next{j}} = \theta_{\next{j}}$, then $x\in\S{j}$ and the inequality reduces to $x_{j} \le \theta_{j}$, which holds true by applying equation~\ref{eq:restrict} to $j$. Otherwise, $x_{\next{j}} \le \theta_{\next{j}} - 1$ and then $x_{j}  + (\ub_{j}-\theta_{j})(x_{\next{j}} - \theta_{\next{j}}) \le x_{j}-\ub_{j}+\theta_{j} \le \theta_{j}$.

Now we show that for any $k\in\Ij$, $x_{j}+\sum_{i\in\Ij:i \le k}\phi_{j}(i)(x_{i}-\theta_{i}) \le \theta_{j}$ is valid for $\S{k}$. From the previous claim, the result is true for $k=\next{j}$. Assume it to be true for some $k\in\Ij$ and consider the inequality for $\S{\next{k}}$.  For $x\in\S{\next{k}}$, equation~\ref{eq:restrict} gives us $x_{\next{k}} \le \theta_{\next{k}}$. If $x_{\next{k}} = \theta_{\next{k}}$, then $x\in\S{k}$ and the inequality reduces to $x_{j}+\sum_{i\in\Ij:i \le k}\phi_{j}(i)(x_{i}-\theta_{i}) \le \theta_{j}$, which is valid for $\S{k}$ from induction hypothesis. Otherwise, $x_{\next{k}} \le \theta_{\next{k}} - 1$ and then
\begin{equation*}
x_{j} - \theta_{j}+\sum_{i\in\Ij:i \le k}\phi_{j}(i)(x_{i}-\theta_{i}) + \phi_{j}(\next{k})(x_{\next{k}}-\theta_{\next{k}}) 
\:\le\: \ub_{j}-\theta_{j}  +\sum_{i\in\Ij:i \le k}\phi_{j}(i)(\ub_{i}-\theta_{i}) - \phi_{j}(\next{k})
\:=\: 0,
\end{equation*}
where the inequality is due to $\phi_{j}(\cdot)\ge 0$ and $x\le\ub$ and the equality follows from Observation~\ref{obs:phi}. This completes the induction process and our proof.
\end{proof}

\green{Since $\theta_{j}=\ub_{j}$ implies $\phi_{j}(i) = 0 \ \forall i\in\Ij$}, it follows that \eqref{packineq} reduces to $x_{j} \le \ub_{j}$ when $\theta_{j}=\ub_{j}$. Thus the only nontrivial packing inequalities are those corresponding to $\theta_{j} < \ub_{j}$. 


\addtocounter{examp}{-1}
\blue{\begin{examp}[continued]\label{examp1:cont}
Recall $\kp = \{x\in\ints^{5}_{+}\colon 2x_{1} + 8x_{2} + 46x_{3} + 150x_{4} + 310x_{4} \le 841, x \le (3,5,2,1,2)\}$ with $\theta=(0,3,1,1,2)$. We have $\phi_{1}(2)=3,\phi_{1}(3) = 9, \phi_{1}(4)=\phi_{1}(5)=18,\phi_{2}(3)=2,\phi_{2}(4)=\phi_{2}(5)=4,\phi_{3}(4)=\phi_{3}(5)=1,\phi_{4}(5)=0$. For $j\in\{1,2,3\}$, our maximal pack inequalities are 
\[
x_{1} + 3x_{2}+9x_{3}+18x_{4}+18x_{5}\le 72, \quad x_{2} + 2x_{3} + 4x_{4} + 4x_{5} \le 17, \quad x_{3} + x_{4} + x_{5} \le 4.
\]
The pack inequalities for $j=4,5$ are the upper bounds $x_{4}\le 1$ and $x_{5}\le 2$.\hfill\myqed
\end{examp}}

Under some additional assumptions on $\wt$ and $\ub$ along with the superincreasing property, one might be able to show that the packing inequality~\eqref{packineq} is a strengthened integer cover \green{or pack} inequality of \citet{atamturk2005cover}. Our proof of Proposition~\ref{prop:valid} is direct, self-contained \green{and motivated from the greedy solution}.

When $\ub = \onevec$, we argue that \eqref{packineq} reduces to a minimal cover of the $\bin$ superincreasing knapsack. \green{Since $\I = \{i\in N\colon \theta_{i}=1\}$, inequality \eqref{packineq} becomes $x_{j} \le 1$ for $j\in \I$}. For $j\notin\I$, we have \green{$\phi_{j}(i) = 1$ $\forall i\in\Ij$} and \eqref{packineq} becomes $x_{j} + \sum_{i\in\Ij}x_{i} \le |\Ij|$. The minimal covers can be obtained from \citep[Theorem 2.4]{laurent1992characterization}: this theorem provides a set of integers $\kappa_{1},\dots,\kappa_{q}$ for some $q\ge 1$ such that $\kappa_{q} = n$ and for any $i < q$, $\kappa_{i} := \max\{t < \kappa_{i+1}\colon \sum_{l=i+1}^{n}\wt_{\kappa_{l}}+\wt_{t} \le \rhs\}$. Proposition~\ref{prop:packing} gives us $\theta_{l}=1$ if and only if $\sum_{k=l+1}^{n}\wt_{k}\theta_{k} < \rhs$. Hence $\kappa_{i} = \max\{t < \kappa_{i+1} \colon 0 < \wt_{t} \le \rhs - \sum_{l=t+1}^{n}\wt_{l}\theta_{l}\} = \max\{t < \kappa_{i+1} \colon \theta_{t} = 1\}$. It follows that $\{\kappa_{1},\dots,\kappa_{q}\} = \I$. Theorem 2.5 in \citep{laurent1992characterization} states that any minimal cover is of the form $j \cup \{\kappa_{i} \colon \kappa_{i} > j \} = j \cup \Ij$, for some $j \notin \{\kappa_{1},\dots,\kappa_{q}\} = \I$. Thus, the minimal cover inequalities are of the form $x_{j} + \sum_{i\in\Ij}x_{i} \le |\Ij|, \: \forall j\notin\I$, which is exactly the same as the packing inequalities.

The next result proves that the packing inequalities can be used to reformulate $\kp$. \green{Later on in Theorem~\ref{thm:conv}, we will prove that they also give an ideal formulation for $\kp$}.
\begin{prop}\label{prop:reform}
$\kp = \left\{x\in\rect \colon x_{j} + \sum_{i\in\Ij}\phi_{j}(i)(x_{i}-\theta_{i}) \, \le \, \theta_{j}, \; j=1,\dots,n\right\}$.
\end{prop}
\begin{proof}
\renewcommand{\gamma}{x}
The forward inclusion $(\subseteq)$ is obvious due to the validity of inequalities~\eqref{packineq}. Consider $\gamma\in\rect$ satisfying all the inequalities~\eqref{packineq}. Clearly, $\gamma=\theta$ is a valid choice that belongs to $\kp$. Suppose that $\gamma\neq\theta$ and define $\i := \max\{i\in N\colon \gamma_{i}\neq\theta_{i}\}$. Since $\gamma$ satisfies the inequality for $j=\i$, $\gamma_{i}=\theta_{i}$ $\forall i > \i$ implies  $\gamma_{\i} \le \theta_{\i}$. Then $\gamma_{\i}\neq\theta_{\i}$ gives us $\gamma_{\i}\le\theta_{\i}-1$ and $\gamma\preceq\theta$. Finally, Proposition~\ref{prop:lexord2} leads to $\gamma\in\kp$.
\end{proof}

\renewcommand{\ite}{i_{t}}
\renewcommand{\face}[1][j]{F_{#1}}

We now show that the packing inequalities also define nontrivial facets of $\cokp$. \green{To aid our arguments}, for every $j\in N$, we define $\xi_{j}\colon\Re^{n}\mapsto\Re$ as 
\[
\xi_{j}(x) := x_{j} - \theta_{j} + \sum_{i\in\Ij}\phi_{j}(i)(x_{i}-\theta_{i}),
\]
\green{and the face defined by this inequality is $\face := \{x\in\cokp\colon \xi_{j}(x) = 0\}$. The integer points on this face have the following properties.}

{
\renewcommand{\ite}{i}
\begin{prop}\label{prop:facetpt3}
Let $\x\in\rect$ be such that for some $j\in N$ with $\theta_{j} < \ub_{j}$ and $\ite\in\Ij$, we have $x_{k} = \theta_{k}$ for all $k\in\Ij$ with $k > \ite$.
\begin{enumerate}
\item If $\x_{j}=\ub_{j}$, $\x_{k}=\ub_{k}$ for all $k\in\Ij$ with $k < \ite$ and $\x_{\ite} = \theta_{\ite}-1$, then $\x \in \face$. 
\item \blue{If $\x_{\ite} \le \theta_{\ite}-2$, then $\x \notin \face$.}
\end{enumerate}
\end{prop}
\begin{proof}
For the first part, $\xi_{j}(\x)  = \ub_{j}-\theta_{j} + \sum_{k\in\Ij\colon k<\ite}\phi_{j}(k)(\ub_{k}-\theta_{k}) -\phi_{j}(\ite)
 =  0$, where the equality is due to Observation~\ref{obs:phi}. \blue{Now suppose that $\x_{\ite}\le\theta_{\ite}-2$. Then 
\[
\xi_{j}(\x) \le x_{j}-\theta_{j} + \sum_{k\in\Ij\colon k < \ite}\phi_{j}(k)(x_{k}-\theta_{k}) - 2 \phi_{j}(\ite)
\le \ub_{j}-\theta_{j} + \sum_{k\in\Ij\colon k < \ite}\phi_{j}(k)(\ub_{k}-\theta_{k}) - 2 \phi_{j}(\ite)
 = -\phi_{j}(\ite),
\]
where the last equality is due to Observation~\ref{obs:phi}. Now $\theta_{j}<\ub_{j} \implies \phi_{j}(\ite) > 0 \implies \xi_{j}(\x) < 0$}.
\end{proof}
\blue{
Choosing $\ite=n$ in Proposition~\ref{prop:facetpt3} yields the following inclusion that will be useful later in \textsection\ref{sec:case1}:}
\begin{equation}\label{facetextpts}\blue{
\face\cap\ints^{n} \subseteq \{x\in\rect\colon x_{n}\in\{\theta_{n}-1,\theta_{n}\}\} \qquad \forall j\in N \text{ such that } \theta_{j}<\ub_{j}.}
\end{equation}
}
\begin{prop}[Facets]\label{prop:facet}
For any $j\in N$ with $\theta_{j} < \ub_{j}$, inequality~\eqref{packineq} is facet-defining to $\cokp$.
\end{prop}
\begin{proof}
For $j\in N$ with $\theta_{j}<\ub_{j}$, we construct $n$ affinely independent points of $\kp$ that belong to $\face\cap\ints^{n}$. These $n$ points can be divided into three categories.
\begin{enumerate}
\item Fix $\x = (0,\ldots,0,\theta_{j}, \theta_{j+1},\ldots,\theta_{n})\preceq\theta$. Clearly $\xi(\x) = 0$. 

\item Fix $\x = (\onevec[l],\ub_{j},0,\ldots,0,\theta_{\next{j}}-1,\theta_{\next{j}+1},\ldots,\theta_{n})\preceq\theta$ for some $l < j$. Here $\xi(\x) = \ub_{j}-\theta_{j} - (\ub_{j}-\theta_{j})= 0$. 

\renewcommand{\ite}{i}
\item Choose $\ite\in\Ij$. There are two subtypes here: (a) fix $\x = (\zerovec,\ub_{j},\ldots,\ub_{\prev{\ite}},\zerovec, \theta_{\ite}-1,\theta_{\ite+1},\ldots,\theta_{n})$, (b) for some $l$ such that $\max\{j,\prev{\ite}\} < l <\ite$, fix $\x = (\zerovec,\ub_{j},\ldots,\ub_{\prev{\ite}},\onevec[l], \theta_{\ite}-1,\theta_{\ite+1},\ldots,\theta_{n})$. Both these points satisfy $\x\preceq\theta$ by construction and are in $\face$ due to Proposition~\ref{prop:facetpt3}.
\end{enumerate}
\newcommand{\M}{M}
We have constructed a total of $1 + j-1 + \next{j} - j + \sum_{i\in\Ij\colon i > \next{j}}(i - \prev{i}) = n$ points in $\face$. Suppose that these $n$ points form $n$ columns of a matrix $\M$ in a way that the columns are sorted as Type 1, then Type 2, and then Type 3 (first all points of subtype (a) and then all of subtype (b)). Let there exist some weights $\lambda_{1},\dots,\lambda_{n}$ such that $\M\lambda = \zerovec, \onevec^{\top}\lambda=0$.

\begin{claim}\label{claim:facet1}
$\lambda_{p} = 0$ for all $p\notin\{j+1,\dots,j+|\Ij|\}$. Consider the $l^{th}$ row of $\M$ for some $1\le l < j$. There is exactly one column of $\M$, corresponding to a Type 2 point, that contains a nonzero entry in row $l$. Hence $\lambda_{2}=\dots=\lambda_{j}=0$. Next consider some $k \notin\I\colon k > j$. There is exactly one column of $\M$, corresponding to a Type 3 subtype (b) point, that contains a nonzero entry in row $k$. Hence all the $\lambda$'s for Type 3 subtype (b) columns are zero, i.e. $\lambda_{j+|\Ij|+1}=\dots=\lambda_{n}=0$. The only remaining nonzero values are for $\lambda_{1},\lambda_{j+1},\dots,\lambda_{j+|\Ij|}$, which must sum to zero. Consider the $j^{th}$ row. Exactly one column of $\M$ has entry $\theta_{j}$ in row $j$ (Type 1 point) while all other columns have entry $\ub_{j}$. This gives us $\lambda_{1}\theta_{j} + \ub_{j}\sum_{p=j+1}^{j+|\Ij|}\lambda_{p} = 0$. Since $\lambda_{1} + \sum_{p=j+1}^{j+|\Ij|}\lambda_{p} = 0$ and $\theta_{j} < \ub_{j}$ by assumption, it follows that $\lambda_{1} = \sum_{p=j+1}^{j+|\Ij|}\lambda_{p} = 0$.~\myqed
\end{claim}

\begin{claim}\label{claim:facet2}
$\lambda=\zerovec$. Consider the $\next{j}^{th}$ row in $\M$. The first Type 3 point has an entry $\theta_{\next{j}}-1$ in this row while all other Type 3 points have an entry of $\ub_{\next{j}}$. The two equalities $\lambda_{j+1}(\theta_{\next{j}}-1) +  \sum_{p=j+2}^{j+|\Ij|}\lambda_{p}\ub_{\next{j}} = 0$ and $\sum_{p=j+1}^{j+|\Ij|}\lambda_{p} = 0$ imply $\lambda_{j+1}(\theta_{\next{j}} - 1 - \ub_{\next{j}}) = 0$, thereby giving us $\lambda_{j+1}=0$ since $ \theta_{\next{j}} \le \ub_{\next{j}}$. Now let $\ite\in\Ij\setminus \{\next{j}\}$. Let the Type 3 subtype (a) point corresponding to $\ite$ be in the $(j+t)^{th}$ column of $\M$ with the associated weight $\lambda_{j+t}$. Assume as part of induction hypothesis that $\lambda_{j+1} = \dots = \lambda_{j+t-1} = 0$. We argue that $\lambda_{j+t}=0$. Observe that the entry for the $\ite^{th}$ row of $\M$ in columns $j+t,j+t+1,\dots,j+|\Ij|$ is $\theta_{\ite}-1,\ub_{\ite},\dots,\ub_{\ite}$, respectively. Upon using the induction hypothesis and $\lambda_{j+t}+\sum_{p=j+t+1}^{j+|\Ij|}\lambda_{p}=0$ in $(j+t)^{th}$ row of $\M\lambda=\zerovec$, we get $\lambda_{j+t}(\theta_{\ite}-1-\ub_{\ite}) = 0$, thereby giving us $\lambda_{j+t}=0$ since $ \theta_{\ite} \le \ub_{\ite}$. This completes the induction process and we have $\lambda_{p}=0, p =j+1,\dots,j+|\Ij|$. Finally, $\lambda=\zerovec$ follows from Claim~\ref{claim:facet1}.\qquad\myqed
\end{claim}

We have shown in Claim~\ref{claim:facet2} that $\lambda=\zerovec$ is the only possible solution to $\M\lambda=\zerovec, \onevec^{\top}\lambda=0$. Hence the $n$ points constructed above are affinely independent and $\face$ is a facet of $\cokp$.
\end{proof}

\renewcommand{\ite}[1][t]{j}
\green{Having shown that the packing inequalities~\eqref{packineq} are strong valid inequalities for $\cokp$, we now prove in Theorem~\ref{thm:conv} that $\cokp$ does not have any other nontrivial facets. Our proof uses the dynamic program of Proposition~\ref{prop:dp} and Figure~\ref{fig:dp}. For $j\in\I$, recall $\ell_{j}$, a subset of feasible solutions at the leaf child of $j$, from equation \eqref{leafnode}. We know from equation \eqref{optsol} that optimal solutions can only be found at leaves of the tree in Figure~\ref{fig:dp}}. If an optimal solution occurs at leaf $\ell_{j}$, i.e. $\ell_{j}\cap\Opt\neq\emptyset$, we say that $j$ is an optimal non-leaf node that is parent to the optimal leaf $\ell_{j}$. While comparing two leaves $\ell_{i}$ and $\ell_{i^{\prime}}$, we say that $\ell_{i}$ is \emph{larger} than $\ell_{i^{\prime}}$ if and only if $i > i^{\prime}$.

\green{Before proving Theorem~\ref{thm:conv}, we present some useful characterizations of the optimal solutions of this dynamic program that will be invoked at multiple points in our proof. We will need the following notation}: let the optimal value in Proposition~\ref{prop:dp} be stated as $\val[\ite] = \max\{{\val[\ite]}^{-}, {\val[\ite]}^{+}\} \ \forall j\in\I$, where
\[
{\val[\ite]}^{-} = [c_{\ite}]^{+}(\theta_{\ite}-1) + \sum_{i=1}^{\ite-1}[c_{i}]^{+}\ub_{i}, \quad {\val[\ite]}^{+} = c_{\ite}\theta_{\ite} + \val[\prev{j}].
\]
It follows that ${\val[\ite]}^{-} + \sum_{i\in\Ij}c_{i}\theta_{i} = \max\{c^{\top}x \colon x \in \ell_{j}\}$. \green{Observe that since $x\le\ub$, then $c_{\ite} = 0$ for some $\ite\in\I$ implies that ${\val[\ite]}^{-} \ge {\val[\ite]}^{+}$. The next two observations are straightforward from the dynamic program of Figure~\ref{fig:dp}.}

\begin{obs}\label{obs:negcost}
For any $\ite\in\I$, $\ell_{\ite}\cap\L\neq\emptyset \blue{\implies} {\val[\ite]}^{-} \ge {\val[\ite]}^{+}$ and ${\val[\ite]}^{-} > {\val[\ite]}^{+} \implies \L \subseteq \ell_{\ite}\cup\bigcup_{i\in\Ij[\ite]}\ell_{i}$.
\end{obs}

\begin{obs}\label{obs:leafval}
Let $j\in\I$ and consider $x\in\ell_{j}\cap\L$. Then for any $i< j$, we have \emph{(i)} $x_{i}=\ub_{i}$ if $c_{i} > 0$, \emph{(ii)} $x_{i}=0$ if $c_{i} < 0$, and \emph{(iii)} $x_{i}$ is unrestricted if $c_{i}=0$.
\end{obs}

\green{The next result gives a sufficient condition for \emph{all} the optimal solutions to lie in the facet defined by the $j^{th}$ packing inequality.}
\begin{lem}\label{lem:face}
Let $j\in N$ be such that $c_{j} > 0$, $c_{t} > 0 \ \forall t\in\Ij$ and if $j\in\I$, we also have $\ell_{j}\cap\L=\emptyset$. Then $\Opt\subseteq\face$. 
\end{lem}
\begin{proof}
Recall that $\Opt \subseteq \{\theta\} \cup \bigcup_{i\in\I} \ell_{i}$. We must show $\xi_{j}(x)=0 \ \forall x\in\Opt$. $\xi_{j}(\theta)=0$ is trivial. Now consider $x\in\ell_{i}\cap\L$ for some $i\in\I$. Our assumption $\ell_{j}\cap\L=\emptyset$ (in case $j\in\I$) means that either $i\in\Ijm$ or $i\in\Ij$. By construction of $\ell_{i}$ in \eqref{leafnode}, we have $x_{t}=\theta_{t} \ \forall t\in\Ij[i]$. If $i\in\Ijm$, then $x_{j}=\theta_{j}$ irrespective of whether $j\in\I$ or not, and subsequently we have $\xi_{j}(x) = 0$. Now let $i\in\Ij$. Since we assumed $c_{j} > 0$, $c_{t} >0 \ \forall t\in\Ij$, \eqref{leafnode} gives us $x_{i}=\theta_{i}-1$ and Observation~\ref{obs:leafval} gives us $x_{j}=\ub_{j}$, $x_{t}=\ub_{t} \ \forall t\in\Ij\colon t < i$. This along with $x_{t}=\theta_{t}\ \forall t\in\Ij[i]$ and Proposition~\ref{prop:facetpt3} gives us $\xi_{j}(x) = 0$.
\end{proof}

We are now ready to prove our first main result.

\begin{thm}\label{thm:conv}
$\cokp\;=\; \left\{x\in[\zerovec,u]\colon\;\; x_{j} \:+\:\sum_{i\in\Ij}\phi_{j}(i)(x_{i}-\theta_{i}) \;\le\;\theta_{j},\;\; \forall j\in N \right\}$.
\end{thm}
\begin{proof}
\newcommand{\U}[1][j]{\mathbb{U}_{#1}}
\newcommand{\T}[1][j]{\mathbb{0}_{#1}}
Let $\face = \{x\in\cokp\colon \xi_{j}(x) = 0\}$, $\U = \{x\in\cokp\colon x_{j}=\ub_{j}\}$, and $\T = \{x\in\cokp\colon x_{j}=0\}$ denote the faces of $\cokp$ defined by the proposed inequalities. Note that $\face[n] = \U[n]$ because $\theta_{n} = \ub_{n}$ and $\Ij[n]=\emptyset$. Based on \citet[Approach 6]{wolsey1998ip}, we must show that for any $c\neq\zerovec$, there exists $j\in N$ such that either $\Opt\subseteq\face$ or $\Opt\subseteq\U$ or $\Opt\subseteq\T$. If there exists $j\in N$ with $c_{j}<0$, then clearly $\Opt\subseteq\T$. Assume $c \ge \zerovec$.

First suppose that $\ell_{n}\cap\L=\emptyset$. Let $\i := \next{\max\{j\in\I\colon \ell_{j}\cap\L\neq\emptyset \}}$ be the smallest non-leaf node that is larger than the parent of every optimal leaf node. Then for any $x\in\Opt$, we have $x_{j}=\theta_{j} \ \forall j\in\i\cup\Ij[\i]$ and it follows that $\xi_{j}(x) = 0$ and hence $\Opt\subseteq\face$ for all $j\in\i\cup\Ij[\i]$. Henceforth assume $\ell_{n}\cap\L\neq\emptyset$. Recall that $\ell_{0}=\{\theta\}$.

\begin{description}
\item[Case i.] $\theta\notin\L$. Let $\i\in\I$ be the parent node of the smallest optimal leaf. We first argue that there exists $i < \i$ such that $c_{i} > 0$. Suppose $c_{i}=0 \ \forall i < \i$. Then ${\val[\i]}^{-} = c_{\i}(\theta_{\i}-1) \le c_{\i}\theta_{\i} = {\val[\i]}^{+}$ and hence $\val[\i] = {\val[\i]}^{+} = c_{\i}\theta_{\i}$. Then every leaf $\ell_{i}$, for all $i < \i$ (including $\ell_{0}$), is optimal, a contradiction to the optimality of $\ell_{\i}$. Hence there exists some $i < \i$ such that $c_{i} > 0$. Since $\ell_{\i}$ is the smallest optimal leaf and $i < \i$, Observation~\ref{obs:leafval} implies that $x_{i}=\ub_{i} \ \forall x\in\Opt$ and thus $\Opt\subseteq\U[i]$. 


\item[Case ii.] $\theta\in\L$. 
\begin{claim}\label{claim:sumtheta}
$\val[\ite] = {\val[\ite]}^{+} = c_{\ite}\theta_{\ite} +\sum_{i\in\Ijm[\ite]}c_{i}\theta_{i}$ for all $\ite\in\I$. Since $\theta\in\L$, we have $\val[i_{1}] = {\val[i_{1}]}^{+} = c_{i_{1}}\theta_{i_{1}}$. Let $j\in\I\setminus\{i_{1}\}$. For any $i\in\Ijm$, Observation~\ref{obs:negcost} and $\theta\in\L$ give us ${\val[i]}^{-} \le {\val[i]}^{+}$. This implies $\val[i] = {\val[i]}^{+} = c_{i}\theta_{i} + \val[\prev{i}]$. Since we already argued $\val[i_{1}] = c_{i_{1}}\theta_{i_{1}}$, a straightforward induction argument gives us the desired claim.\qquad\myqed
\end{claim}
It follows that ${\val[\ite]}^{+} \ge {\val[\ite]}^{-}$ for all $j\in\I$.

\begin{description}
\item[Case ii-a.] $\exists j\notin\I$ such that $c_{j} > 0$. We first argue that this case leads to $c_{i} > 0 \ \forall i\in\Ij$. Suppose that $c_{i}=0$ for some $i\in\Ij$. Since $i > j$ with $j\notin\I$, we have $j\notin\Ijm[i]$. Consider the following:
\[
\blue{{\val[i]}^{-} - {\val[i]}^{+} \:=\: \sum_{t < i}c_{t}\ub_{t} - \sum_{t\in\Ijm[i]}c_{t}\theta_{t} \:=\: \sum_{t\in\Ijm[i]}c_{t}(\ub_{t}-\theta_{t}) \:+\: \sum_{t< i\colon t \notin\,\Ijm[i]\cup j}c_{t}\ub_{t} \:+\: c_{j}\ub_{j} \:>\: 0, }
\]
\blue{where the strict inequality is due to $\theta\le\ub,\ub>\zerovec,c\ge\zerovec$ and $c_{j}>0$. Thus we have arrived at a contradiction to ${\val[i]}^{+} \ge {\val[i]}^{-}$.} Hence $c_{i} > 0 \ \forall i\in\Ij$. Applying Lemma~\ref{lem:face} gives us $\Opt\subseteq\face$.

\item[Case ii-b.] $c_{i}=0 \ \forall i\notin\I$ or $\I = \{1,\dots,n\}$.  Since $c \neq 0$ and $c \ge \zerovec$, there exists some $i\in\I$ with $c_{i} > 0$. First suppose $c_{i} > 0$ for all $i\in\I$. Since $\val[i_{1}] = c_{i_{1}}\theta_{i_{1}}$ by Claim~\ref{claim:sumtheta}, we have $\ell_{i_{1}}\cap\L=\emptyset$ and Lemma~\ref{lem:face} gives us $\Opt\subseteq\face[i_{1}]$. Now let $\i := \max\{i\in\I \colon c_{i} = 0\}$ be the largest non-leaf node with cost coefficient equal to zero. By construction, $c_{i} > 0$ for all $i\in\Ij[\i]$. The identity $\val[\i] = \max\{\sum_{i\in\Ijm[\i]}c_{i}\ub_{i}, \sum_{i\in\Ijm[\i]}c_{i}\theta_{i} \}$ along with $\val[\i] = \sum_{i\in\Ijm[\i]}c_{i}\theta_{i}$ from Claim~\ref{claim:sumtheta} and $c_{\i}=0$ implies that
\begin{equation}\label{casecondition}
c_{i} = 0 \quad \text{OR} \quad \left(c_{i} > 0 \;\text{  and  }\; \theta_{i}=\ub_{i}\right) \qquad \forall i \in\Ijm[\i].
\end{equation}
\qquad First suppose there exists a $j\in\Ijm[\i]$ with $c_{j} > 0$. Equation~\eqref{casecondition} gives us $\theta_{j}=\ub_{j}$ and along with $\Ijm\subseteq\Ijm[\i]$, also implies $\sum_{t\in\Ijm}c_{t}(\ub_{t}-\theta_{t})=0$. \blue{Now ${\val[j]}^{+} = c_{j} + \sum_{t\in\Ijm}c_{t}\theta_{t} > \sum_{t\in\Ijm}c_{t}\ub_{t} = {\val[j]}^{-}$ and as a result, Observation~\ref{obs:negcost} implies} $\ell_{j}\cap\L=\emptyset$. Consider an optimal solution $x\in\ell_{t}\cap\Opt$ for some $t\in\{0\}\cup(\I\setminus\{j\})$. If $t > j$, then $c_{j} > 0$ and Observation~\ref{obs:leafval} implies $x_{j}=\ub_{j}$. Otherwise $t < j$ and $x_{j}$ is fixed to $\theta_{j} = \ub_{j}$. Hence $\Opt\subseteq\U$ if there exists a $j\in\Ijm[\i]$ with $c_{j} > 0$.  

\qquad Finally, suppose that $c_{i} = 0 \ \forall i\in\Ijm[\i]$, or $\i = i_{1}$ and $\Ijm[\i]=\emptyset$. Since we have already assumed in this case that $c_{i}=0 \ \forall i\notin\I$, it follows that $c_{i}=0 \ \forall i < \i$. Hence $\i < n$, because otherwise $c=\zerovec$. Consider $\next{\i}$, the first non-leaf node above $\i$. The definition of $\i$ gives us $c_{\next{\i}} > 0$ and $c_{i} > 0 \ \forall i\in\Ij[\next{\i}]$. \green{We argue that $\ell_{\next{\i}}\cap\L=\emptyset$; doing so and invoking Lemma~\ref{lem:face} would lead to $\Opt\subseteq\face[\next{\i}]$.} Claim~\ref{claim:sumtheta} gives us $\val[\next{\i}] = {\val[\next{\i}]}^{+} = c_{\next{\i}}\theta_{\next{\i}}$ and since $c_{\next{\i}} > 0$, we have ${\val[\next{\i}]}^{+} > c_{\next{\i}}(\theta_{\next{\i}}-1)$. \green{Now $c_{i}=0 \ \forall i < \i$ implies that ${\val[\next{\i}]}^{-} = c_{\next{\i}}(\theta_{\next{\i}}-1)$. Hence ${\val[\next{\i}]}^{+} > {\val[\next{\i}]}^{-}$ and Observation~\ref{obs:negcost} implies that $\ell_{\next{\i}}\cap\L=\emptyset$}. 
\end{description}
\end{description}
All the above cases are mutually exclusive and exhaustive. Hence our proof is complete.
\end{proof}

\subsection{Applications of Theorem~\ref{thm:conv}}\label{sec:appln1}
\paragraph{Lower bounded knapsack}
\newcommand{\lb}{l}
\newcommand{\lokp}{\ensuremath{\kpset^{\lb}}}
For a lower bounded superincreasing knapsack $\lokp:= \{x\in[\lb,\ub]\cap\ints^{n}_{+}\colon \wt^{\top}x \le \rhs\}$, we can (i) perform a variable change $y= x - \lb$ to obtain $\kp = \{y\in\rect[][\ub-\lb]\colon \wt^{\top}y\le\rhs-\wt^{\top}\lb\}$, (ii) apply Theorem~\ref{thm:conv} to get $\cokp$, and (iii) substitute back $x=y + \lb$ to obtain $\conv{\lokp}$. In particular, it is straightforward to verify that if $\lb\le\theta$, then $\conv{\lokp} = (\cokp)\cap[\lb,\ub]$.

\paragraph{Divisible knapsack}
An integer basis is a strictly increasing sequence $\{\wt_{i}\}_{i\ge 1} \subset \ints_{++}$  with the property that there exists a sequence $\{\ub_{i}\}_{i\ge 1} \subset \ints_{++}$ such that every $\rhs\in\ints_{++}$ can be expressed as $\rhs = \sum_{i=1}^{n}\wt_{i}x_{i}$ for some $n$ and $x\in\rect$. An equivalent characterization due to Cantor \citep[cf.][Theorem 2.1]{pitteloud2004log} is the following: \emph{$\{\wt_{i}\}$ is an integer basis if and only if $\wt_{1}=1$ and $\wt_{i}\,|\,\wt_{i+1} \ \forall i$. Moreover, the sequence $\{\ub_{i}\}$ is uniquely determined as $\ub_{i}=\frac{\wt_{i+1}}{\wt_{i}} - 1$}. Then, any finite subsequence of an integer basis $\{\wt_{i}\}$ and its corresponding $\{\ub_{i}\}$ define a divisible superincreasing knapsack, whose convex hull is given by Theorem~\ref{thm:conv}. The set $\expkp$ introduced in \eqref{expkp} is a particular case that uses powers of $\base$ as its integer basis. Another class of divisible superincreasing knapsacks is obtained when $\ub_{i}=\ub_{0} \ \forall i$ and some $0 < \ub_{0} \le \frac{1}{2}\min_{i}\frac{\wt_{i+1}}{\wt_{i}}$.

\newcommand{\oneub}{\ub^{1}}
\newcommand{\twoub}{\ub^{2}}
\newcommand{\rrhs}{d}
\newcommand{\wwt}{w}
\newcommand{\onekp}{\kp^{\le}} 
\newcommand{\twokp}{\kp^{\ge}} 
\newcommand{\T}{T}
\newcommand{\Tj}[1][j]{\T_{#1}}
\newcommand{\g}{g}
\newcommand{\h}{h}
\renewcommand{\cokp}{\conv{\onekp}}
\newcommand{\cotwokp}{\conv{\twokp}}
\renewcommand{\i}{s}
\newcommand{\M}{N}

\section{Intersection of knapsacks}\label{sec:twoside}
In this section, we consider the problem of convexifying the intersection of $m\ge 2$ superincreasing knapsacks of $\le$- or $\ge$-types. We prove that $\bigO(n)$ number of linear inequalities describe the convex hull of this intersection. The number of inequalities is independent of the number of intersecting knapsack sets. It suffices to  address the case of \emph{two} intersecting  knapsacks; the general case follows immediately after noting that every superincreasing knapsack corresponds to a lexicographically ordered set of integer vectors and the lexicographic order is a total order. Our proof generalizes a recent result for $\bin$ superincreasing knapsacks by \citet{muldoon2012}.

Note that if we are given two $\le$-type superincreasing knapsacks -- $\{x\in\rect\colon \wt^{\top}x\le\rhs\}$ with maximal packing $\gamma$ and $\{x\in\rect\colon\wwt^{\top}x\le\rrhs \}$ with maximal packing $\theta$, and w.o.l.o.g. we assume that $\gamma\preceq\theta$, then Proposition~\ref{prop:lexord2} tells us that their intersection is equal to $\{x\in\rect\colon x \preceq\gamma\}$. Hence the convex hull of the intersection of two $\le$-type knapsacks is given by $\bigO(n)$ packing inequalities corresponding to one of the sets. The nontrivial case to prove is when we are intersecting a $\le$-type and a $\ge$-type knapsack.

Henceforth, let $\onekp :=\{x\in\rect\colon \wt^{\top}x\le\rhs\}$ and $\twokp := \{x\in\rect\colon\wwt^{\top}x\ge\rrhs \}$ be two superincreasing knapsacks with $\wt,\wwt >\zerovec$ and $\rhs,\rrhs > 0$. Proposition~\ref{prop:lexord2} implies that 
\begin{equation}\label{interlex}
\onekp\cap\twokp = \{x\in\rect\colon \gamma\preceq x \preceq \theta\},
\end{equation}
where $\theta$ is the maximal packing of $\kp$ given by \eqref{greedysol} and $\gamma$ is the minimal packing of $\twokp$ obtained from \eqref{greedysol} by complementing variables:
\[
\gamma_{i} := \ub_{i}\: - \: \min\left\{\ub_{i}, \: \left\lfloor \frac{\wwt^{\top}\ub - \rrhs - \sum_{k=i+1}^{n}\wwt_{k}(\ub_{k}-\gamma_{k})}{\wwt_{i}}\right\rfloor\right\} \quad \forall i=n,\ldots,1.
\]
It follows from Theorem~\ref{thm:conv} that
\begin{equation}\label{convtwokp}
\cotwokp \:=\: \left\{x\in[\zerovec,\ub]\colon x_{j} + \sum_{i\in\Tj}\Phi_{j}(i)(x_{i}-\gamma_{i}) \,\ge\, \gamma_{j} \;\; \forall j\in N \right\},
\end{equation}
where $\T = \{i\in N\colon \gamma_{i}\le\ub_{i}-1\},\Tj = \{i\in\T\colon i > j\}$, and $\Phi_{j}(i) = \gamma_{j}\prod_{k\in\Tj\colon k < i}(\gamma_{k}+1)$ for all $i\in\Tj$. The main result of this section proves that the convex hull operator distributes over $\onekp\cap\twokp$.
\green{
\begin{thm}\label{thm:twokp2}
$\conv{\{x\in\rect\colon \gamma \preceq x \preceq \theta \}} = \conv{\{x\in\rect\colon x \succeq \gamma \}}\, \cap\, \conv{\{x\in\rect\colon x \preceq \theta \}}$. In particular, if $\wt,\wwt > \zerovec$, then $\conv{(\onekp\cap\twokp)} = \cokp\,\cap\,\cotwokp$. 
\end{thm}
}
This is an interesting result because in general for any two arbitrary sets $\mathcal{X}_{1}$ and $\mathcal{X}_{2}$, we have $\conv{(\mathcal{X}_{1}\cap\mathcal{X}_{2})} \subseteq \conv{\mathcal{X}_{1}} \cap \conv{\mathcal{X}_{2}}$. For the intersection $\mathcal{X} = \cap_{i=1}^{m_{1}}\onekp_{i}\, \bigcap\, \cap_{i=1}^{m-m_{1}}\twokp_{i}$ of $m\ge 2$ superincreasing knapsacks (each having a coefficient vector of strictly positive integers), the equivalence to lexicographic ordering implies that $\mathcal{X} = \onekp_{\i}\,\cap\,\twokp_{\i^{\prime}}$ for some indices $\i,\i^{\prime}$. Then Theorem~\ref{thm:twokp2} gives us $\bigO(n)$ inequalities to describe the convex hull of this intersection.
\blue{
\begin{rem}
For $\onekp\cap\twokp$, the assumption $\wt,\wwt > \zerovec$ is not w.o.l.o.g since we are considering two knapsacks simultaneously. Suppose that $\wt_{i} > 0$ for all $i\in\M_{1}:=\{1,\ldots,n_{1}\}$ and some $1\le n_{1}\le n-1$, and $\wwt_{i}>0 $ for all $i\in\M_{2}\subseteq N$ (assume w.o.l.o.g. that $\{n_{1}+1,\ldots,n\}\subseteq\M_{2}$). In this case, we don't have the distributive property and in general, $\conv{(\onekp\cap\twokp)} \subsetneq \cokp\,\cap\,\cotwokp$, as shown by the following example. 
\begin{examp}
Let $\onekp = \{x\in\ints^{7}_{+}\colon 2x_{1} + 8x_{2} + 46x_{3} + 150x_{4} + 310x_{4} \le 841, x \le (3,5,2,1,2,4,2)\}$ and $\twokp = \{x\in\ints^{7}_{+}\colon 2x_{4} + 7x_{5} + 30x_{6}+50x_{7} \ge 150, x \le (3,5,2,1,2,4,2)\}$. The inequalities describing $\cokp$ are described in Example~\ref{examp1:cont} in \textsection\ref{sec:conv} whereas the nontrivial facets of $\cotwokp$ are $x_{5} + 2x_{6} + 4x_{7}\ge 12$, $ x_{6}+x_{7} \ge 4$ and $x_{7}\ge 1$. The PORTA software \citep{christofporta} tells us that the intersection of $\cokp$ and $\cotwokp$ has a fractional extreme point $(0, 0, 2, 1, 1, 7/2, 1)$.\hfill\myqed
\end{examp}
\noindent In fact, we claim that $\onekp\cap\twokp$ may not be equal to the set of integer points that are lexicographically ordered between two given integer vectors. Suppose it were true: $\onekp\cap\twokp = \{x\in\rect\colon \gamma^{\prime}\preceq x \preceq \theta^{\prime}\}$ for some $\gamma^{\prime},\theta^{\prime}\in\ints^{n}_{+}$. Then, because $\onekp = \{x\in\rect\colon (x_{i})_{i\in\M_{1}} \preceq (\theta_{i})_{i\in\M_{1}} \}$ and $\twokp = \{x\in\rect\colon (x_{i})_{i\in\M_{2}} \succeq (\gamma_{i})_{i\in\M_{2}} \}$, we must have $\theta^{\prime}_{i}=\theta_{i}\ \forall i\in\M_{1}, \theta^{\prime}_{i}= \ub_{i} \ \forall i\in N\setminus\M_{1}, \gamma^{\prime}_{i} = \gamma_{i} \ \forall i\in\M_{2},\gamma^{\prime}_{i}= 0 \ \forall i\in N\setminus\M_{2}$. It is obvious that $\onekp\cap\twokp \subseteq \{x\in\rect\colon \gamma^{\prime}\preceq x \preceq \theta^{\prime}\}$. Since the knapsack $\sum_{i\in\M_{1}}\wt_{i}x_{i}\le\rhs$ is nontrivial (i.e. $\sum_{i\in\M_{1}}\wt_{i}\ub_{i}>\rhs$), there must exist some $k\in\{1,\ldots,n_{1}-1\}$ such that $\theta_{k} < \ub_{k}$. Now, $x^{\prime} := (\gamma^{\prime}_{1},\ldots,\gamma^{\prime}_{k-1},\max\{\theta_{k}+1,\gamma^{\prime}_{k}\},\gamma^{\prime}_{k+1},\ldots,\gamma^{\prime}_{n-1},\theta_{n}-1)$ satisfies $\gamma^{\prime}\preceq x^{\prime}\preceq\theta^{\prime}$ but $(x^{\prime}_{i})_{i\in\M_{1}} \npreceq (\theta_{i})_{i\in\M_{1}}$. In fact, $x^{\prime} \notin \cokp$. This gives us a contradiction. Thus, in the presence of zeros in the coefficients of at least one of the two knapsacks, we cannot use the nice structural properties of lexicographic orderings to convexify the intersection of $\onekp$ and $\twokp$.
\end{rem}}

The rest of this section is devoted to proving Theorem~\ref{thm:twokp2}. We assume throughout that $\wt,\wwt > \zerovec$ and hence identity~\eqref{interlex} holds true. In order to ensure that $\onekp\cap\twokp$ is a full-dimensional set, we assume w.o.l.o.g. that $\sum_{i=1}^{n-1}\wwt_{i}\ub_{i} \:+\: (\theta_{n}-1)\wwt_{n} \ge \rrhs$ and $\gamma_{n}\le\theta_{n}-1$; otherwise we can fix $x_{n}=\theta_{n}$ and address the lower dimensional case with modified right hand sides. \green{Our arguments are divided into two cases: $\gamma_{n} \le \theta_{n}-2$ and $\gamma_{n}=\theta_{n}-1$. The proof of the first case depends on a geometric intuition, as explained in \textsection\ref{sec:case1}. This geometric insight breaks down when $\gamma_{n}=\theta_{n}-1$ and hence we resort to some technical lemmas in \textsection\ref{sec:case2}. Based on these building blocks, the proof of Theorem~\ref{thm:twokp2} is presented in \textsection\ref{sec:proof}.
}
%

\subsection{$\gamma_{n}\le\theta_{n}-2$}\label{sec:case1}
\blue{
Consider Figure~\ref{fig:cases}. It is apparent that in the two-dimensional case, we always have $\conv\{x\in\rect[2]\colon \gamma\preceq x \preceq \theta\} = \conv\{x\in\rect[2]\colon x \succeq\gamma\}\cap\{x\in\rect[2]\colon x \preceq \theta \}$. In Figure~\ref{fig:case1}, where $\gamma_{2} \le \theta_{2}-2$, we see that $\conv\{x\in\rect[2]\colon \gamma\preceq x \preceq \theta\}$ is equal to $(A \cup B)\bigcup (B\cup C)$, where the three sets are defined as follows: $A=\{x\in\conv{\{y\in\rect[2]\colon y\preceq\theta\}}, x_{2}\ge\theta_{2}-1\}$, $B = [0,\ub_{1}]\times [\gamma_{2}+1 ,\theta_{2}-1]$ and $C=\{x\in\conv{\{y\in\rect[2]\colon y\succeq\gamma\}}, x_{2}\le \gamma_{2}+1\}$. This geometric intuition of expressing the convex hull as a union of two sets enables us to prove that the convex hull operator distributes for arbitrary $n$ when $\gamma_{n}\le\theta_{n}-2$.
}
\begin{figure}[htbp]
\begin{center}
\subfigure[$\gamma_{n} \le \theta_{n}-2$. Convex hull is the union of $A,B,C$.]{
\label{fig:case1}
\includegraphics[scale=0.4]{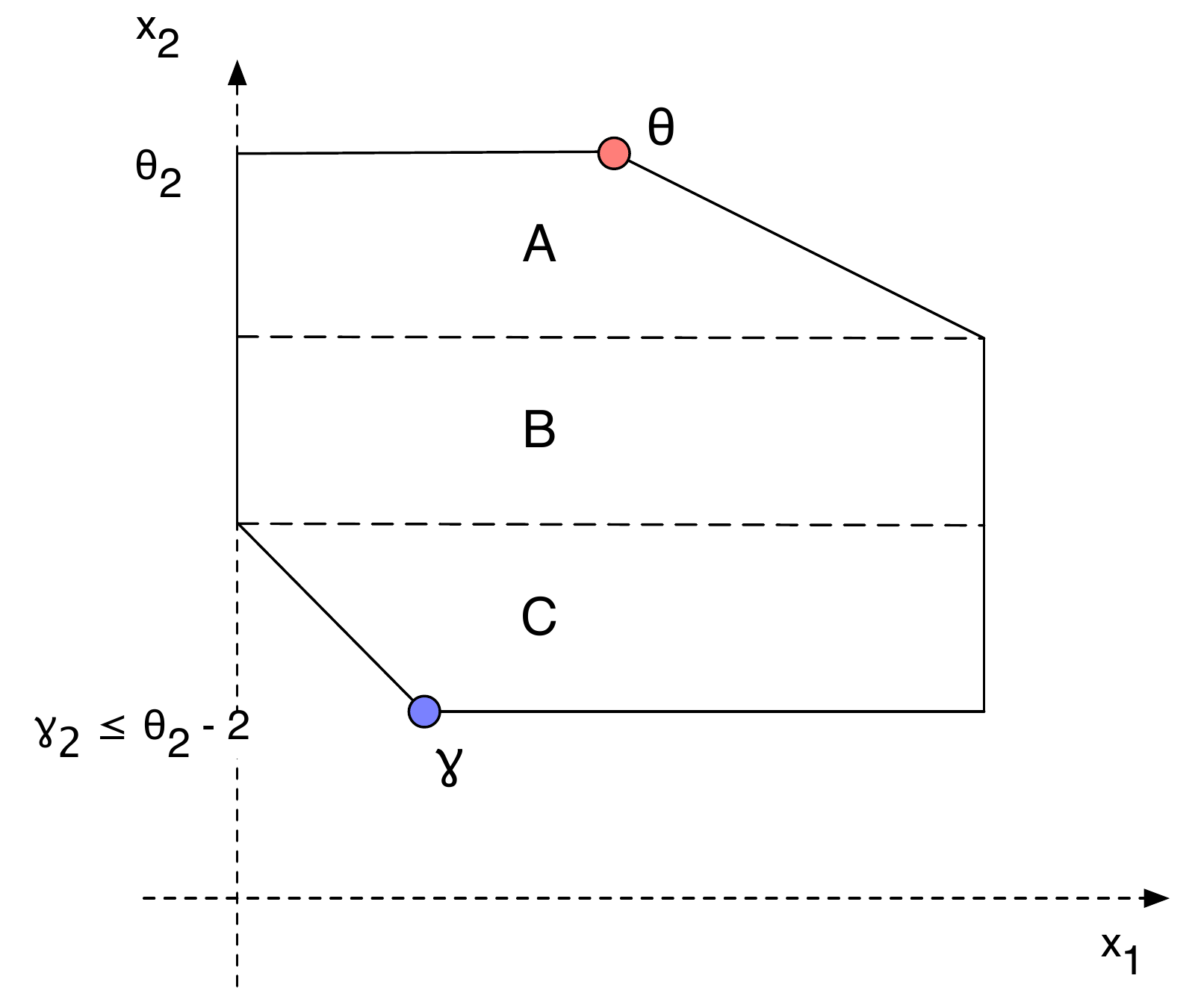}
}\hspace{0.5cm}
\subfigure[$\gamma_{n}=\theta_{n}-1$. Convex hull is the union of $\{x\in\onekp\colon x_{2}=\theta_{2}\}$ and $\{x\in\twokp\colon x_{2}=\theta_{2}-1\}$.]{
\label{fig:case2}
\includegraphics[scale=0.4]{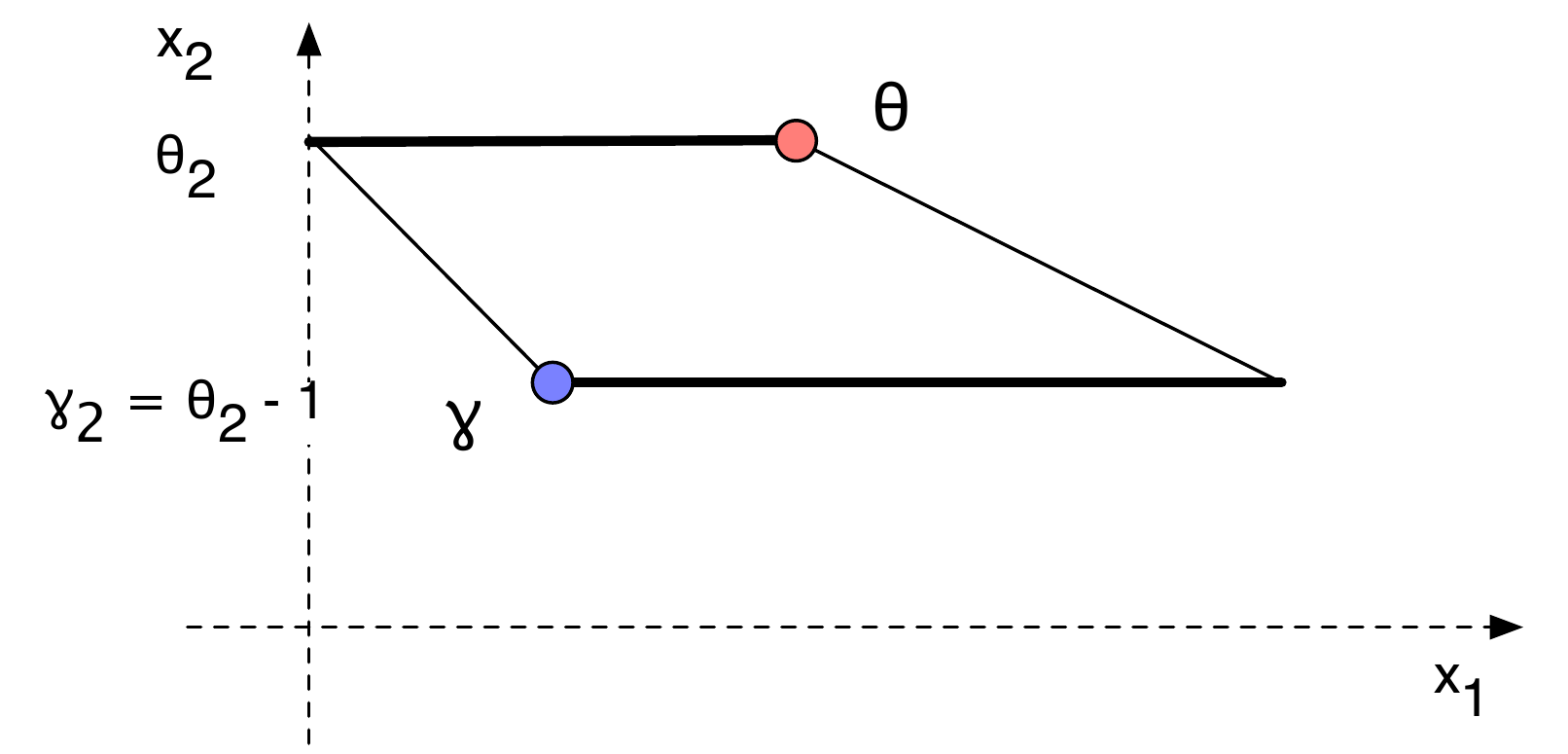}
}
\caption{Two cases for convexifying $\{x\in\rect\colon\gamma\preceq x \preceq \theta\}$.}
\label{fig:cases}
\end{center}
\end{figure}

\blue{
\begin{prop}\label{prop:case1}
Assume that $\gamma_{n}\le\theta_{n}-2$. Then 
\[
\conv{(\onekp\cap\twokp)} = \{x\in\cokp\colon x_{n}\ge \gamma_{n}+1 \} \cup \{x\in\cotwokp\colon x_{n}\le\theta_{n}-1 \} = \cokp\cap\cotwokp.
\]
\end{prop}
\begin{proof}
$\conv{(\onekp\cap\twokp)} \subseteq \cokp\cap\cotwokp$ is obvious. Take $y\in\cokp\cap\cotwokp$. If $y_{n}\ge\gamma_{n}+1$ then $y\in\cokp\cap\{x\colon x_{n}\ge\gamma_{n}+1\}$; otherwise $y_{n} < \gamma_{n}+1$ and the assumption $\gamma_{n}+1\le\theta_{n}-1$ implies that $y\in\cotwokp\cap\{x\colon x_{n}\le\theta_{n}-1\}$. Now let $y$ be an extreme point of $\cokp\cap\{x\colon x_{n}\ge\gamma_{n}+1\}$. Theorem~\ref{thm:conv} and equation~\eqref{facetextpts} imply that $y\in\onekp$ with $y_{n}\in\{\gamma_{n}+1,\theta_{n}-1,\theta_{n}\}$. Since $\gamma_{n}+1\le\theta_{n}-1$, it follows that $\gamma \preceq y \preceq \theta$ and identity \eqref{interlex} establishes $y \in \conv{(\kp\cap\twokp)}$. The arguments for $y\in\cotwokp\cap\{x\colon x_{n}\le\theta_{n}-1\}$ are similar. Thus $\{x\in\cokp\colon x_{n}\ge \gamma_{n}+1 \} \cup \{x\in\cotwokp\colon x_{n}\le\theta_{n}-1 \}\subseteq \conv{(\onekp\cap\twokp)}$, thereby completing our proof.
\end{proof}
The above proof heavily relies on the assumption $\gamma_{n}\le\theta_{n}-2$. In particular, if $\gamma_{n}=\theta_{n}-1$, then we can only show that $\{x\in\cokp\colon x_{n}\ge \gamma_{n}+1 \} \cup \{x\in\cotwokp\colon x_{n}\le\theta_{n}-1 \} \subseteq \cokp\cap\cotwokp$ but cannot argue the $\supseteq$-inclusion using the above steps. Hence the case $\gamma_{n}=\theta_{n}-1$ requires a different proof technique, which is presented next.
}

\subsection{$\gamma_{n}=\theta_{n}-1$}\label{sec:case2}
\green{We start by writing a disjunctive representation of $\onekp\cap\twokp$, somewhat similar in vein to the first equality in Proposition~\ref{prop:case1}}.
\begin{lem}\label{lem:twokpunion}
$\onekp\cap\twokp = \{x\in\onekp\colon x_{n}=\theta_{n} \} \cup \{x\in\twokp \colon \blue{x_{n}=\theta_{n}-1} \}$.
\end{lem}
\begin{proof}
Since $x_{n}\le\ub_{n}=\theta_{n}$ for any $x\in\onekp\cap\twokp$, the elementary disjunction $\{x_{n} \le \theta_{n}-1\}\cup\{x_{n}=\theta_{n}\}$ gives us 
$\onekp\cap\twokp =\{x\in\onekp\cap\twokp \colon x_{n}\le\theta_{n}-1\} \cup  \{x\in\onekp\cap\twokp \colon x_{n}=\theta_{n} \}$.
For any $x\in\rect$ such that $x_{n}\le\theta_{n}-1$, we have $x\prec\theta$ and hence $\onekp\cap\{x\colon x_{n}\le\theta_{n}-1\} = \rect\cap\{x\colon x_{n}\le\theta_{n}-1\}$. Since $\twokp=\{x\in\rect\colon x \succeq \gamma\}$ \blue{and $\gamma_{n}=\theta_{n}-1$}, we get $\onekp\cap\twokp\cap\{x\colon x_{n}\le\theta_{n}-1\} = \twokp\cap\{x\colon \blue{x_{n}=\theta_{n}-1}\}$. Next, note that $\wwt_{n}\theta_{n}\ge \sum_{i=1}^{n-1}\wwt_{i}\ub_{i} + \wwt_{n}(\theta_{n}-1) \ge \rrhs$. Then it follows that $\{x\in\rect\colon x_{n}=\theta_{n}\}\subset\twokp$ and we get $\onekp\cap\twokp \cap \{x\colon x_{n}=\theta_{n}\} = \onekp\cap\{x\colon x_{n}=\theta_{n}\}$.
\end{proof}
Lemma~\ref{lem:twokpunion} will be crucial in completing the proof of this case in \textsection\ref{sec:proof}. \green{The proposed disjunction is depicted in Figure~\ref{fig:case2} for $\Re^{2}$. It is easy to see that the two nontrivial facets obtained by convexifying this union in $\Re^{2}$ are exactly the packing inequalities for $\{x\in\rect[2]\colon x\preceq\theta\}$ and $\{x\in\rect[2]\colon x\succeq\gamma\}$. Motivated by this illustration, our approach is use the extended formulation of \citet{balas1998disjunctive} to convexify the union in Lemma~\ref{lem:twokpunion} and argue that every fractional point in $\cokp\cap\cotwokp$, i.e. $x\in \cokp\cap\cotwokp$ such that $x_{n}=\theta_{n}-\epsilon$ for some $\epsilon\in(0,1)$, belongs to the convex hull of $\onekp\cap\twokp$.  To do so, we must characterize points in the $\epsilon$-restrictions of $\cokp$ and $\cotwokp$. This is achieved in the next lemma}. Recall from Definition~\ref{defn:I} that we denote the support of $\theta$ as $\I = \{i_{1},\ldots,i_{\r},i_{\r+1}=n\}$ for some $\r \ge 0$.

\begin{lem}\label{lem:slice1}
Let $\epsilon\in(0,1)$ and $x\in[\zerovec,\ub]$ such that $x_{n}=\theta_{n}-\epsilon$. Define $\delta_{i} :=\min\{\epsilon\ub_{i},x_{i}\}$ for all $i\le i_{\r}$ and $\delta_{i} := x_{i}$ for all $i > i_{\r}$.
\begin{enumerate}
\item If $x\in\cokp$, then $x_{j} - \delta_{j} - \theta_{j}(1-\epsilon) + \sum_{i\in\Ij\setminus n}\phi_{j}(i)(x_{i}-\delta_{i} - \theta_{i}(1-\epsilon)) \,\le\, 0$ for every $j\in N\setminus n$.
\item If $x\in\cotwokp$, then $\delta_{j} + \sum_{i\in\Tj\setminus n}\Phi_{j}(i)\delta_{i} \,\ge\, \epsilon\Phi_{j}(n)$ for every $j\in N\setminus n$.
\end{enumerate}
\end{lem}
\green{To prove this technical lemma, we need to exploit the recursive nature of $\phi_{j}(\cdot)$ and $\Phi_{j}(\cdot)$ so that we can rearrange expressions suitably. The following lemma gives us the required result; its proof is a tedious algebraic exercise and is hence relegated to \ref{app:one}.}

\renewcommand{\x}{\epsilon}
\renewcommand{\delta}{z}

\begin{lem}\label{lem:facetcoeff1}
Let $j\in N\setminus n$ and $\x \neq 0$. 
\begin{enumerate}
\item For $\i,i\in\Ij$ with $\i < i$, $\phi_{j}(i) =  \phi_{j}(\i)\left[1 + \sum_{\substack{k\in\Ij\\ \i\le k < i}}\phi_{k}(i) \right]$.
\item For $\i \in\Ij\setminus n$, $\sum_{\substack{i\in\Ij\setminus n\\i\ge\i}}\phi_{j}(i)(x_{i}-\delta_{i}-\theta_{i}\x)=\phi_{j}(\i)\,\sum_{\substack{i\in\Ij\setminus n\\i\ge\i}}\left[x_{i}-\delta_{i}-\theta_{i}\x+\sum_{k\in\Ij[i]\setminus n}\phi_{i}(k)(x_{k}-\delta_{k} - \theta_{k}\x) \right]$.
\item For $\i\in\Tj\setminus n$,
$\sum_{\substack{i\in\Tj\setminus n\\i\ge\i}}\Phi_{j}(i)\left(\gamma_{i}\x - \delta_{i}\right) =\Phi_{j}(\i)\,\sum_{\substack{i\in\Tj\setminus n\\i\ge\i}}\left[\x\,\Phi_{i}(n) - \delta_{i} - \sum_{k\in\Tj[i]\setminus n}\Phi_{i}(k)\delta_{k} \right]$.
\end{enumerate}
\end{lem}
\begin{proof}
In \ref{app:one}.
\end{proof}

\begin{proof}[Proof of Lemma~\ref{lem:slice1}]
\green{We prove the first part here; arguments for the second part are analogous and are provided in \ref{app:one} for completeness}. Choose $j\in N\setminus n$. If $\theta_{j} = \ub_{j}$, then the inequality is obvious because $\phi_{j}(\cdot) = 0$. Assume $\theta_{j} < \ub_{j}$. For $j\ge i_{\r}$, the inequality holds because $\Ij\setminus n=\emptyset$ and the $j^{th}$ nontrivial facet in $\cokp \cap \{x\colon x_{n}=\theta_{n}-\epsilon\}$ can be written as $x_{j} - \epsilon\ub_{j} \le \theta_{j}(1-\epsilon)$. Now consider $j \le i_{\r}-1$ and assume that the inequality holds for all $\imath = j+1,\ldots,n-1$. 

\begin{claim}\label{claim:indhypo}
For any $\i\in\Ij\setminus n$, we have $\sum_{i\in\Ij\setminus n\colon i \ge \i}\phi_{j}(i)(x_{i}-\delta_{i}-\theta_{i}(1-\epsilon)) \:\le\: 0$. Follows from Lemma~\ref{lem:facetcoeff1}, $\phi_{j}(\i)\ge 0$ and induction hypothesis. 
\qquad\myqed
\end{claim}

First suppose that $\delta_{j}=\epsilon\ub_{j}$. Let $\i\in\Ij\setminus n$ be the smallest index such that $\delta_{\i}=x_{\i}$. If $\i$ does not exist, then $\delta_{i}=\epsilon\ub_{i} \ \forall i\in\Ij\setminus n$ and upon rearranging, we must show that 
\[
x_{j}\;+\;\sum_{i\in\Ij\setminus n}\phi_{j}(i)(x_{i}-\theta_{i})\;-\;\epsilon\left[\ub_{j}-\theta_{j}\;+\;\sum_{i\in\Ij\setminus n}\phi_{j}(i)(\ub_{i}-\theta_{i}) \right] \;\le\;\theta_{j},
\]
which is exactly the $j^{th}$ nontrivial facet of $\cokp\cap\{x\colon x_{n}=\theta_{n}-\epsilon \}$ due to $\phi_{j}(n) = \ub_{j}-\theta_{j}+\sum_{i\in\Ij\setminus n}\phi_{j}(i)(\ub_{i}-\theta_{i})$ from Observation~\ref{obs:phi}. Now suppose $\i$ exists. Rewriting the desired inequality, we must show that
\[
\frac{x_{j}-\theta_{j}}{\ub_{j}-\theta_{j}}-\epsilon \;\le\; \Lambda \;:=\;-\,\sum_{i\in\Ij\setminus n}\frac{\phi_{j}(i)}{\ub_{j}-\theta_{j}}(x_{i}-\delta_{i}-\theta_{i}(1-\epsilon)).
\]
It suffices to verify that $\Lambda \ge 1-\epsilon$. We split the summation over $\Ij\setminus n$ into three parts: first over all $i < \i$, then $\i$, and third over all $i > \i$, and use the first part of Lemma~\ref{lem:facetcoeff1} to rewrite $\phi_{j}(i)/(\ub_{j}-\theta_{j}) \ \forall i \in \Ij\setminus n$. This gives us
\begin{multline*}
\Lambda\;=\; -\,\sum_{i\in\Ij\colon i < \i}\left[1\,+\,\sum_{k\in\Ij \colon k < i}\phi_{k}(i) \right](x_{i}-\epsilon\ub_{i}-\theta_{i}(1-\epsilon))
\: + \:\left[1\,+\,\sum_{k\in\Ij \colon k < \i}\phi_{k}(\i) \right]\theta_{\i}(1-\epsilon) \\
 \; \:-\, \sum_{i\in\Ij\setminus n\colon i > \i}\left[1\,+\,\sum_{k\in\Ij \colon k < \i}\phi_{k}(i) \,+\,\sum_{\substack{k\in\Ij\colon\\ \i \le k < i}}\phi_{k}(i)\right](x_{i}-\delta_{i}-\theta_{i}(1-\epsilon)).
\end{multline*}
Combining common terms gives us 
\begin{eqnarray*}
\Lambda \:=\: \theta_{\i}(1-\epsilon) - \Lambda_{1} - \Lambda_{2}
\::=\: \theta_{\i}(1-\epsilon) &-&\sum_{i\in\Ij\colon i < \i}\left[x_{i}-\epsilon\ub_{i}-\theta_{i}(1-\epsilon)\,+\,\sum_{t\in\Ij[i]\setminus n}\phi_{i}(t)(x_{t}-\delta_{t}-\theta_{t}(1-\epsilon))\right] \\
&-& \sum_{i\in\Ij\setminus n\colon i > \i}\left[1\,+\, \sum_{\substack{k\in\Ij\colon\\ \i \le k < i}}\phi_{k}(i) \right](x_{i}-\delta_{i}-\theta_{i}(1-\epsilon)).
\end{eqnarray*}
The induction hypothesis implies $\Lambda_{1}\le 0$. First statement of Lemma~\ref{lem:facetcoeff1} implies $\Lambda_{2} = \frac{1}{\phi_{j}(\i)}\sum_{i\in\Ij\setminus n\colon i > \i}\phi_{j}(i)(x_{i}-\delta_{i}-\theta_{i}(1-\epsilon))$. Claim~\ref{claim:indhypo} now gives us $\Lambda_{2}\le 0$. Thus $\Lambda \ge \theta_{\i}(1-\epsilon) \ge 1 -\epsilon$ since $\i\in\Ij$ and $\theta_{\i}\ge 1$.

Finally, let $\delta_{j}=x_{j}$. Then Claim~\ref{claim:indhypo} gives us $-\theta_{j}(1-\epsilon) + \sum_{i\in\Ij\setminus n}\phi_{j}(i)(x_{i}-y_{i}-\theta_{i}(1-\epsilon)) \,\le\, 0$.
\end{proof}

\renewcommand{\P}{\mathcal{P}}
\subsection{Proving the distributive property}\label{sec:proof}
\begin{proof}[Proof of Theorem~\ref{thm:twokp2}]
\green{The case $\gamma_{n}\le\theta_{n}-2$ is proved in Proposition~\ref{prop:case1}}. Now suppose that $\gamma_{n}=\theta_{n}-1$. It remains to show that $\conv{(\onekp\cap\twokp)}\supseteq\cokp\cap\cotwokp$ since the $\subseteq$-inclusion is obvious. We first obtain an extended formulation for $\conv{(\onekp\cap\twokp)}$. For convenience, define
\begin{equation}\label{defgh}
\g_{j} := \theta_{j} + \sum_{i\in\Ij\setminus n}\phi_{j}(i)\theta_{i} \quad \forall j \in N\colon \theta_{j}\le\ub_{j}-1,	\quad
\h_{j} := \gamma_{j} + \sum_{i\in\Tj\setminus n}\Phi_{j}(i)\gamma_{i} \quad \forall j\in N\colon \gamma_{j} \ge 1.
\end{equation}
\begin{claim}\label{claim:twokpext}
$\conv{(\onekp\cap\twokp)}$ is equal to the projection onto the $x$-space of the polytope
\begin{subequations}\label{correq}
\begin{align}
\P &\;\;:=& \Big\{(x,y)\in\Re^{n}_{+}\times\Re^{n}_{+} \colon  &\zerovec\le y \le x, \: \blue{\theta_{n}-1 \le x_{n}\le\theta_{n},\; y_{n} = (\theta_{n}-1)(\theta_{n}-x_{n})},\;\; x_{i} = y_{i}, \quad i_{\r} < i < n	\notag\\ 
&& & y_{i} + \frac{\ub_{i}}{\theta_{n}}(x_{n} - y_{n}) \le \ub_{i}, \quad   i< n, \quad y_{i} - x_{i} + \frac{\ub_{i}}{\theta_{n}}(x_{n} - y_{n}) \ge 0, \quad  i\le i_{\r}	\label{correq4}\medskip\\
&& & x_{j} - y_{j} \: + \: \displaystyle\sum_{i\in\Ij\setminus n}\phi_{j}(i)(x_{i} - y_{i}) \: - \: \frac{\g_{j}}{\theta_{n}}(x_{n}-y_{n}) \;\le \; 0, \quad  j \in N\colon \theta_{j}<\ub_{j}\label{correq5}\medskip\\
&& & y_{j} \:+\: \displaystyle\sum_{i\in\Tj\setminus n}\Phi_{j}(i)y_{i} \:+\:\frac{\h_{j}}{\theta_{n}}(x_{n} - y_{n}) \; \ge \; \h_{j}, \quad  j \in N\colon \gamma_{j} \ge 1\;\Big\}.\label{correq6}
\end{align}
\end{subequations}
The proof of this claim makes use of Lemma~\ref{lem:twokpunion} and the disjunctive programming result of \citet{balas1998disjunctive}; it is provided in \ref{app:one}.\quad\myqed
\end{claim}
Since $\gamma_{n}=\theta_{n}-1$, we get $\conv{\{x\in\twokp\colon x_{n}=\theta_{n}-1\}} = \cotwokp\cap\{x\colon x_{n}=\gamma_{n}\}$. The fact that  $x_{n}\le\theta_{n}$ defines a face of $\conv{\onekp}$ implies that $\conv{\{x\in\onekp\colon x_{n}=\theta_{n}\}} = \cokp\cap\{x\colon x_{n}=\theta_{n}\}$. \blue{Consider $x\in\cokp\cap\cotwokp$. Note that since $x_{n}\le\theta_{n}$ is valid to $\cokp$ and $x_{n}\ge\gamma_{n}=\theta_{n}-1$ is valid to $\cotwokp$, it must be that $x_{n}\in[\theta_{n}-1,\theta_{n}]$. If $x_{n}\in\{\theta_{n}-1,\theta_{n}\}$, then Lemma~\ref{lem:twokpunion} implies $x\in\conv{\left(\onekp\cap\twokp \right)}$}. Let $x_{n}=\theta_{n}-\epsilon$ for some $\epsilon\in (0,1)$. Fix $y\in\Re^{n}_{+}$ as follows: $y_{i} = \min\{\epsilon\ub_{i},x_{i}\}$ for $i \le i_{\r}$, $y_{i}=x_{i}$ for $i_{\r} < i < n$, and $y_{n} = \epsilon(\theta_{n}-1)$. From Claim~\ref{claim:twokpext}, it suffices to show that $(x,y)\in\P$. By construction, $y$ satisfies \eqref{correq4} and the trivial relations with $x$. Since $x_{n}-y_{n} = \theta_{n}(1-\epsilon)$, \eqref{correq5} and \eqref{correq6}, respectively, are transformed to
\begin{subequations}
\begin{eqnarray}
x_{j} - y_{j} + \displaystyle\sum_{i\in\Ij\setminus n}\phi_{j}(i)(x_{i} - y_{i}) &\le& \g_{j}(1-\epsilon) \quad \forall j \in N\colon \theta_{j}\le\ub_{j}-1 \label{newcheck1}\\
y_{j} + \displaystyle\sum_{i\in\Tj\setminus n}\Phi_{j}(i)y_{i} &\ge& \h_{j}\epsilon \quad \forall j \in N\colon \gamma_{j} \ge 1.	 \label{newcheck2}
\end{eqnarray}
\end{subequations}
Since $\h_{j} = \Phi_{j}(n)$ (analogous to Observation~\ref{obs:phi}), inequality \eqref{newcheck2} becomes $y_{j} + \sum_{i\in\Tj\setminus n}\Phi_{j}(i)y_{i} \ge \epsilon \Phi_{j}(n)$. Then Lemma~\ref{lem:slice1} (with $y$ replacing $\delta$) implies that \eqref{newcheck1} and \eqref{newcheck2} are satisfied. Hence $(x,y)\in\P$. 
\end{proof}
\blue{
\begin{rem}
We believe that the proof used for the difficult case $\gamma_{n}=\theta_{n}-1$ can be modified to handle the case $\gamma_{n}\le\theta_{n}-2$ as well. However our geometric arguments in \textsection\ref{sec:case1} lend more intuition into the structural properties of intersection of $\succeq$ and $\preceq$ ordered cones.
\end{rem}
}

\newcommand{\uub}{\tilde{\ub}}
\newcommand{\thetaone}{\theta_{\lfloor\rhs\rfloor}}
\newcommand{\thetatwo}{\theta_{\lfloor \rhs - \uub \rfloor}}
\newcommand{\thetathr}{\gamma_{\lceil \rhs - \uub \rceil}}
\subsection{\blue{Application of Theorem~\ref{thm:twokp2}}}
\blue{
Consider a mixed integer knapsack with a single continuous variable defined by the set $Q := \{(x,y)\in\rect\times[0,\uub]\colon \sum_{i\in N}\wt_{i}x_{i} + y \le \rhs\}$, where we assume that $\{(\wt_{i},\ub_{i})\}_{i\in N}$ forms a superincreasing sequence of tuples of positive integers and $\uub$ and $\rhs$ are positive reals with $\uub\le\rhs$. It is straightforward to verify that $Q = Q_{1}\cup Q_{2}$, where 
\begin{align*}
Q_{1}&:= \{(x,y)\in\rect\times\Re\colon \lceil \rhs-\uub \rceil \le \wt^{\top}x \le \lfloor \rhs \rfloor,\, 0 \le y \le \rhs - \wt^{\top}x\}\\
Q_{2}&:=  \{(x,y)\in\rect\times[0,\uub]\colon \wt^{\top}x \le \lfloor \rhs - \uub \rfloor\}.
\end{align*}
Let $\thetaone$ and $\thetatwo$ denote maximal packings for $\wt^{\top}x\le\lfloor\rhs\rfloor$ and $\wt^{\top}x \le \lfloor \rhs - \uub \rfloor$, respectively, and $\thetathr$ denote a minimal packing for $\wt^{\top}x \ge \lceil \rhs-\uub \rceil$. It follows that $\conv{Q_{2}} =\conv{\{x\in\rect\colon x\preceq\thetatwo \}} \times [0,\uub]$ and hence Theorem~\ref{thm:conv} gives us
\begin{equation}\label{convq1}
\conv{Q_{2}} = \{(x,y)\in[\zerovec,\ub]\times [0,\uub]\colon \text{inequalities \eqref{packineq} for $\thetatwo$} \}.
\end{equation}
For the convex hull of $Q_{1}$, observe that $\rhs - \wt^{\top}x \ge 0$ is valid to $\conv{Q_{1}}$ and hence 
\begin{align}
\conv{Q_{1}} &= \left(\conv{\{(x,y)\in\rect\times\Re\colon \thetathr\preceq x \preceq \thetaone \}}\right)\,\cap\,\{(x,y)\colon 0 \le y \le \rhs - \wt^{\top}x \} \notag\\
&= \left\{(x,y)\in[\zerovec,\ub]\times\Re_{+}\colon \text{inequalities \eqref{packineq} for $\thetaone$},\, \text{inequalities \eqref{convtwokp} for $\thetathr$},\, \wt^{\top}x + y \le \rhs  \right\} \label{convq2}
\end{align}
where the second equality is from Theorem~\ref{thm:twokp2}. Since $Q = Q_{1}\cup Q_{2}$, we have $\conv{Q} = \conv{(\conv{Q_{1}} \cup \conv{Q_{2}})}$. Equations~\eqref{convq1} and \eqref{convq2} and disjunctive programming \citep{balas1998disjunctive} imply a compact extended formulation for $\conv{Q}$.
}

\section{Discussion}\label{sec:discuss}\green{
In this paper, we have identified a special class of general integer knapsacks, referred to as superincreasing knapsacks. We studied its greedy solution $\theta$ and showed that this well-structured set is equal to the set of integer vectors that are lexicographically less ($\preceq$-) than $\theta$. The convex hull of this $\preceq$-ordered set is described using $\bigO(n)$ facets, where $n$ is the dimension of the knapsack, and all the nontrivial facets are derived from $\theta$. An arbitrary knapsack is in general a strict subset of solutions that are $\preceq$-than the greedy solution and hence our facet description yields a class of valid inequalities that can be possibly strengthened by other means for use in cutting plane algorithms. A second interesting phenomenon exhibited by the $\preceq$-ordering and superincreasing structure is that the convex hull operator distributes over a finite intersection. Our results generalize previously known descriptions for $\bin$ superincreasing knapsacks.}

\newcommand{\newlex}{\mathcal{C}}
\blue{
\paragraph{Generalized lexicographic cone}
Finally, we mention that the results derived in this paper can be generalized as follows. Given $\beta\in\ints^{n}_{++}$, let $x\preceq_{\beta} \theta$ denote that $x$ is $\beta$-lex smaller than $\theta$, i.e. either $x=\theta$ or the first index $i$ in reverse order is such that $x_{i} \le \theta_{i}-\beta_{i}$. Suppose that we are interested in convexifying $\newlex :=\{x\in\rect\colon x\preceq_{\beta}\theta\}$. It is easy to verify that if $\beta\neq\onevec$, then $\newlex$ cannot be represented as a integer knapsack; we may need a disjunctive formulation to include the correct set of feasible solutions. We believe that by exploiting the properties of lexicographic orderings, all the results proved in this paper carry through with suitable adjustments; for example the function $\phi_{j}(\cdot)$ in equation \eqref{newphi} must be modified to
$\phi_{j}(i) \::=\: \frac{(\ub_{j}-\theta_{j})}{\beta_{j}}\,\prod_{\substack{k=\next{j}\colon\\k\in\I}}^{\prev{i}}\frac{(\ub_{k}+\beta_{k}-\theta_{k})}{\beta_{k}}$ for all $i\in\Ij$.
}

\bibliographystyle{model1-num-names}
\bibliography{knapsack}

\renewcommand{\i}{s}
\appendix
\section{Missing proofs of \textsection\ref{sec:twoside}}\label{app:one}

\begin{proof}[Proof of Lemma~\ref{lem:facetcoeff1}]
For the first part, note that
$
\frac{\phi_{j}(i)}{\phi_{j}(\i)} \;=\;\prod_{\substack{k\in\Ij\\ \i\le k < i}}(\ub_{k}+1-\theta_{k}) \;=\; 1 \,+\, \sum_{\substack{k\in\Ij\\ \i\le k < i}}(\ub_{k}-\theta_{k})\prod_{\substack{t\in\Ij\\ k < t < i}}(\ub_{t}+1-\theta_{t}) \;=\; 1\,+\,\sum_{\substack{k\in\Ij\\ \i\le k < i}}\phi_{k}(i)
$. We prove the second statement by induction on $|\{i\in\Ij\setminus n\colon i\ge\i\}|$. The third statement can be proven similarly via induction on $|\{i\in\Tj\setminus n\colon i \ge \i\}|$. The claim is clearly true when the cardinality is 1. Assume it is true when the cardinality is $m \ge 1$ and let $|\{i\in\Ij\setminus n\colon i\ge\i\}|=m+1$. For convenience, denote $\{i\in\Ij\setminus n\colon i\ge\i\} = \{\i_{1}:=\i,\i_{2},\dots,\i_{m+1}\}$. Then the left hand side in the lemma is
\begin{multline*}\label{simplexpr}
\sum_{i\in\{\i_{1},\dots,\i_{m}\}}\phi_{j}(i)(x_{i}-\delta_{i}-\theta_{i}\x)\;+\;\phi_{j}(\i_{m+1})(x_{\i_{m+1}}-\delta_{\i_{m+1}}-\theta_{\i_{m+1}}\x) \\
\;=\;\phi_{j}(\i)\,\sum_{i\in\{\i_{1},\dots,\i_{m}\}}\left[x_{i}-\delta_{i}-\theta_{i}\x\:+\:\sum_{k\in\Ij[i]\setminus \{n,\i_{m+1}\}}\phi_{i}(k)(x_{k}-\delta_{k}-\theta_{k}\x) \right] \;+\;\phi_{j}(\i_{m+1})(x_{\i_{m+1}}-\delta_{\i_{m+1}}-\theta_{\i_{m+1}}\x)	
\end{multline*}
where the equality is obtained by invoking the induction hypothesis on the first term. Substituting $\phi_{j}(\i_{m+1}) = \phi_{j}(\i)\left[1 + \sum_{t=1}^{m}\phi_{\i_{t}}(\i_{m+1}) \right]$ from the first part into the above equality and combining common terms, we get
\begin{eqnarray*}
\sum_{\substack{i\in\Ij\setminus n\\i\ge\i}}\phi_{j}(i)(x_{i}-\delta_{i}-\theta_{i}\x) \;=\;
\phi_{j}(\i)\,\sum_{i\in\{\i_{1},\dots,\i_{m+1}\}}\left[x_{i}-\delta_{i}-\theta_{i}\x\:+\:\sum_{k\in\Ij[i]\setminus n}\phi_{i}(k)(x_{k}-\delta_{k}-\theta_{k}\x) \right],
\end{eqnarray*}
which is the desired result.
\end{proof}

\begin{proof}[Proof of second part of Lemma~\ref{lem:slice1}]
If $\gamma_{j}=0$, then $\Phi_{j}(\cdot) = 0$ and again the inequality is obvious. For $j > i_{\r}$, we have $\delta_{i}=x_{i} \ \forall i \in j\cup\Tj\setminus n$. Then $x\in\cotwokp \cap \{x\colon x_{n}=\theta_{n}-\epsilon\}$ implies $x_{j} + \sum_{i\in\Tj\setminus n}\Phi_{j}(i)x_{i} \ge \gamma_{j} + \sum_{i\in\Tj}\Phi_{j}(i)\gamma_{i} - \Phi_{j}(n)(\theta_{n}-\epsilon)$ and the right hand side can be simplified to $\epsilon\Phi_{j}(n)$ (analogous to Observation~\ref{obs:phi}) and $\gamma_{n}=\theta_{n}-1$. Now consider $j\le i_{\r}$ and assume that the inequality holds for all $\imath =j+1,\dots,n-1$. Let $\i\in\Tj\setminus n$ be the smallest index such that $\delta_{\i}=\epsilon\ub_{\i}$ and suppose that $\i$ exists. Upon rearranging terms, we have to check that
\[
\min\{\epsilon\ub_{j},x_{j}\} +\sum_{i\in\Tj\colon i<\i}\Phi_{j}(i)x_{i} \;\ge\; \Lambda \;:=\;\epsilon\Phi_{j}(n) - \Phi_{j}(\i)\epsilon\ub_{\i} - \sum_{i\in\Tj\colon \i < i < n}\Phi_{j}(i)\delta_{i}.
\]
It suffices to verify that $\Lambda\le 0$. 
\begin{claim}\label{claim:hj}
For any $j\in N\setminus n$, we have 
$\Phi_{j}(n) = \gamma_{j} + \sum_{i\in\Tj\setminus n}\Phi_{j}(i)\gamma_{i} 
= \Phi_{j}(k)(\gamma_{k}+1) + \sum_{i\in\Tj[k]\setminus n}\Phi_{j}(i)\gamma_{i}$ for all $ k\in\Tj\setminus n$. The first equality is analogous to Observation~\ref{obs:phi}. The second statement follows from a straightforward reverse induction on $k$ and using the fact that $\Phi_{j}(\next{k})=\Phi_{j}(k)(\gamma_{k}+1)$.\qquad\myqed
\end{claim}

Using Claim~\ref{claim:hj} with $k=\i$, we rewrite $\Phi_{j}(n)$ to get $\Lambda = \epsilon\Phi_{j}(\i)(\gamma_{\i}+1-\ub_{\i}) + \sum_{i\in\Tj\colon \i < i < n}\Phi_{j}(i)(\epsilon\gamma_{i} - \delta_{i})$. Applying Lemma~\ref{lem:facetcoeff1} with $\x = \epsilon$ gives us
\[
\Lambda \;=\; \epsilon\Phi_{j}(\i)(\gamma_{\i}+1-\ub_{\i}) \:+\: \Phi_{j}(\next{\i})\sum_{\substack{i\in\Tj\\ \i < i < n}}\left[\epsilon\Phi_{i}(n) - \delta_{i} - \sum_{k\in\Tj[i]\setminus n}\Phi_{i}(k)\delta_{k} \right]
\;\le\;\epsilon\Phi_{j}(\i)(\gamma_{\i}+1-\ub_{\i}) \;\le\; 0
\]
where $\next{\i} = \min\{i\colon i\in\Tj\}$, the first inequality is due to each summand being non-positive from induction hypothesis and the second inequality is due to $\i\in\Tj$ and hence $\gamma_{\i}\le\ub_{\i}-1$.

If $\i$ does not exist then $y_{i}=x_{i} \ \forall i \in\Tj\setminus n$ and we must show that $\min\{\epsilon\ub_{j},x_{j} \} + \sum_{i\in\Tj\setminus n}\Phi_{j}(i)x_{i} \ge \epsilon\Phi_{j}(n)$. If $x_{j} \le \epsilon\ub_{j}$, then the same argument as that used for $j > i_{\r}$ proves the desired inequality. Otherwise $\epsilon\ub_{j} < x_{j}$. In this case, we set $\Lambda = \epsilon\Phi_{j}(n) - \sum_{i\in\Tj\setminus n}\Phi_{j}(i)x_{i}$, rewrite $\Phi_{j}(n) = \gamma_{j} + \sum_{i\in\Tj\setminus n}\Phi_{j}(i)\gamma_{i}$ as in Claim~\ref{claim:hj} and follow same steps as before to obtain $\Lambda \le \epsilon\gamma_{j}\le \epsilon\ub_{j}$, as desired.
\end{proof}

\begin{proof}[Proof of Claim~\ref{claim:twokpext}]
We have $\conv{(\onekp\cap\twokp)} =\conv{\left(\conv \{x\in\onekp\colon x_{n}=\theta_{n} \} \cup \conv\{x\in\twokp \colon x_{n}=\theta_{n}-1 \} \right)}$ from Lemma~\ref{lem:twokpunion}. Since $\gamma_{n}=\theta_{n}-1$, we get $\conv{\{x\in\twokp\colon x_{n}=\theta_{n}-1\}} = \cotwokp\cap\{x\colon x_{n}=\gamma_{n}\}$. The fact that  $x_{n}\le\theta_{n}$ defines a face of $\conv{\onekp}$ implies that $\conv{\{x\in\onekp\colon x_{n}=\theta_{n}\}} = \cokp\cap\{x\colon x_{n}=\theta_{n}\}$.
Applying \citeauthor{balas1998disjunctive}' result and invoking Theorem~\ref{thm:conv} and equation~\eqref{convtwokp} gives us the following extended formulation for $\conv{(\onekp\cap\twokp)}$:
\begin{equation*}
\begin{array}{lcll}
\P^{\prime} &=& \Big\{(x,y,\lambda)\in\Re^{n}_{+}\times\Re^{n}_{+}\times [0,1] \colon & y_{n} = \gamma_{n}(1-\lambda),\: \zerovec \le y \le \ub(1-\lambda), \; \zerovec \le x-y \le \ub\lambda, \\
&&& x_{n}-y_{n}=\theta_{n}\lambda \ \;\; x_{i} - y_{i} = 0 \quad i_{\r} < i < n,\\
&&& x_{j}-y_{j}\,+\,\displaystyle\sum_{i\in\Ij\setminus n}\phi_{j}(i)(x_{i}-y_{i}) \:\le\: \g_{j}\lambda, \quad  j \le i_{\r}\colon\theta_{j}\le\ub_{j}-1\\ 
&&& y_{j}\,+\,\displaystyle\sum_{i\in\Tj\setminus n}\Phi_{j}(i)y_{i} \:\ge\:\h_{j}(1-\lambda),  \quad  j\in N\colon \gamma_{j} \ge 1\Big\}. 
\end{array}
\end{equation*}
where $\g_{j}$ and $\h_{j}$ are defined in \eqref{defgh}.
The equality $x_{n}-y_{n}=\theta_{n}\lambda$ implies $\lambda = (x_{n}-y_{n})/\theta_{n}$. Upon substituting for $\lambda$ in $\P^{\prime}$ and rearranging the inequalities, we get the proposed claim.
\end{proof}

\end{document}